\documentclass{article}
\usepackage{graphicx} % Required for inserting images
\input{macro.tex}

\DeclareMathAlphabet{\mathpzc}{OT1}{pzc}{m}{it}
\usepackage{tikz-cd}

\title{Power Operations in Morava E-Theory of Flat Ring Spectra}
\author{Yuval Lotenberg}

\begin{document}

\maketitle

\begin{abstract}

    Let $E_n$ be Morava $E$-theory of height $n$. Let $R$ be a $p$-adically flat commutative ring spectrum. Then the Tate-valued Frobenius map endows $\pi_0 R$ with the structure of a $\delta$-ring.  On the other hand, we may form the $K(n)$-completed tensor product $L_{K(n)}(R \otimes E_n)$, which is a $K(n)$-local $E_n$-algebra. Then $\pi_0(L_{K(n)}(R \otimes E_n)) = LT_n \cotimes \pi_0 R$ admits the structure of an algebra over the monad $\TT(n)$ defined by Rezk. The $\TT(n)$-algebra structure encodes the power operations of $L_{K(n)}(R \otimes E_n)$.
    
    In this paper we describe the $\TT(n)$-algebra structure on $\pi_0(L_{K(n)}(R \otimes E_n))$.
\end{abstract}

\setcounter{tocdepth}{1}
\tableofcontents

\section{Introduction} \label{section:intro}

\subsection{Background and Main Results}

Let $p$ be a fixed prime. A connective $p$-complete spectrum $R$ is called $p$-adically flat if its $\FF_p$-homology, $\FF_p \otimes R$, is discrete (i.e., concentrated at homotopical degree $0$). 

If $R$ is $p$-adically flat, then $\pi_0R$ is endowed with a natural $\delta$-ring structure.
This arises from the Tate-valued Frobenius map $\phi\colon R \to R^{tC_p}$: composing $\phi$ with the inverse of the trivial map $R \to R^{tC_p}$, which is an equivalence for $p$-adically flat spectra, induces an endomorphism $\pi_0R \to \pi_0R$, which is a lift of the Frobenius morphism. Since $\pi_0 R$ is $p$-torsion free, lifts of Frobenius are in bijection with $\delta$-ring structures. 
(see \cref{subsection:delta-rings} for details). 
\\\\
Let $E:=E(n)$ denote Morava $E$-theory of height $n$, and let $K(n)$ be Morava $K$-theory of height $n$. If $R$ is a $K(n)$-local $E$-algebra, then it admits power operations relative to $E$. These operations have been extensively studied (see, for example, \cite{Rezk09}, \cite{BF15}, \cite{Nissan25}).
In particular, Rezk defines a monad $\TT$ for which $\pi_0(R)$ is naturally a $\TT$-algebra. $\TT$-algebras have an underlying $E_0$-algebra, which in the case of a $K(n)$-local $E(n)$-algebra $R$ is $\pi_0 R$. 

Moreover, Rezk constructs a functor from the category of $\TT$-algebras to sheaves on a certain deformation category (see \ref{subsection:formal-groups}), and identifies its essential image when restricted to $p$-torsion-free $\TT$-algebras. Explicitly, this is a subcategory $\Sh(\Def, \Alg)_{\text{cong}}$ of sheaves satisfying a certain congruence criterion. 
\\ \\
If $R$ is a $p$-adically flat ring spectrum, we may form the completed tensor product $L_{K(n)}(E \otimes R) = E \cotimes R$, which is a $K(n)$-local $E$-algebra. Consequently, $\pi_0(E \cotimes R)$ carries a natural $\TT$-algebra structure.

These constructions fit into the following diagram:
\begin{equation} \label{diag:dotted-functor}
\begin{tikzcd}[row sep=large, column sep = large] 
    \CAlg_{/\pcs}^{\text{flat}} \arrow[r,  "L_{K(n)}(- \otimes E(n))"] \arrow[d, "\pi_0"] &[4em] \CAlg_{E(n)}^{K(n),\text{even}} \arrow[d, "\pi_0"] \\ 
    \Alg_{\delta, \text{tf},p\text{-complete}} \arrow[r, dashed, "?"] & \Alg_{\TT,\text{tf}}
\end{tikzcd}
\end{equation}

where \emph{tf} means torsion-free.
A natural question that arises from the digagram above is whether the dotted arrow exists, namely, whether the $\TT$-algebra structure of $\pi_0(R \cotimes E)$ can be recovered from the $\delta$-ring structure of the $\pi_0 R$, where $R$ is a $p$-adically flat ring spectrum.
If so, it is desirable to find an explicit algebraic description of this missing functor from $\delta$-rings to $\TT$-algebras.

\begin{example}
In the case of height $n = 1$, taking $\Gamma = \GG_m$ the multiplicative formal group, $E(1)$ identifies with $p$-completed complex $K$-theory. In this case, $\TT$-algebras are equivalent to $\delta$-rings.

Explicitly, recall that $\TT$-algebras correspond to sheaves on the deformation category, satisfying a certain congruence condition.
For any ring $R$ deforming $\FF_p$, there is a unique (up to isomorphism) deformation of $\GG_m$ to $R$, namely $\GG_m$ over $R$. There is also a unique deformation of $\Frob$ between any two deformations. Therefore, the data of a sheaf $B$ on the deformation category, reduces to a a (coherent) choice of $R$-algebra $A_R = B_R(\GG_m)$ for every ring $R$, in addition to an endomorphism $B_R(\Frob)=\psi : A \to A$. By coherence, $A_R = R \cotimes A_{\ZZ_p}$, thus it is enough to understand the case of $A = A_{\ZZ_p}$. When we substitute $R = \FF_p$, we see that the congruence condition requires that $\psi / p =: A/ p \to A/p$ is the Frobenius map, i.e., $\psi$ is a lift of the Frobenius. Finally, in the torsion free case, lifts of Frobenius are in bijection with $\delta$-ring structures.

This suggests that the missing dotted functor may simply be the identity functor.
\end{example}

In this paper, we give an explicit algebraic description of the functor $\Alg_{\delta, \text{tf},p\text{-complete}} \to \Alg_{\TT,\text{tf}}$, using Rezk’s theory identifying $\Alg_{\TT, \text{tf}}$ with sheaves on the deformation category. This is given in \cref{section:construction}. In particular, we obtain

\begin{thm}
    Let $(R,\delta)$ be a $p$-complete torsion-free $\delta$-ring, with associated lift of Frobenius $\psi(x) = x^p+ p\delta(x)$ for $x \in R$. 
    Let $B \in \Sh(\Def, \Alg)$ be image of $R$ under the composition of functors 
    \[
    \Alg_{\delta, \text{tf},p\text{-complete}} \to \Alg_{\TT,\text{tf}} \xrightarrow[]{\sim} \Sh(\Def, \Alg)_{\text{tf,cong}}
    \]
   Then for a deformation $G$ over $S$, we have $B_S(G) = S \cotimes R$ (completed with respect to the maximal ideal of $S$), and to a morphism $f : G \to G'$ deforming $\Frob^r$, the sheaf assigns $B_S(f) = \Id_{S} \otimes \psi^r$. 

   On maps of $\delta$-rings, the functor just gives completed tensor product with $E_0$ -- maps of $\TT$-algebras are just maps of underlying rings. 
\end{thm}

And in particular,

\begin{thm}
    Let $R$ be a $p$-complete, $p$-adically flat commutative ring spectrum, and let $\psi :\pi_0R \to \pi_0R$ be the lift of Frobenius associated to its $\delta$-ring structure.
    
    Let $B \in \Sh(\Def, \Alg)$ denote the sheaf associated to the $\TT$-algebra structure on $L_{K(n)}(R \otimes E)$. This sheaf assigns to a deformation $G$ over $S$ the ring $B_S(G)=S \cotimes \pi_0R$. To a morphism $f\colon G \to G'$ deforming $\Frob^r$, the sheaf gives $B_S(f)=\Id_S \otimes \psi^r$.
\end{thm}

Moreover, in height 1, for a $p$-complete $\delta$-ring $R$, the associated sheaf $B$ assigns $B_{\ZZ_p}(G)=R$, and to the unique deformation of Frobenius, it assigns the lift of Frobenius coming from the $\delta$-ring structure of $R$. Hence we get the following corollary:

\begin{cor}
    In height $1$, the missing dotted functor from diagram~(\ref{diag:dotted-functor}) is the identity.
\end{cor}

If $S$ is a $K(n)$-local $E$-algebra, then  $\pi_0(S)$ has an associated sheaf $B^{\pi_0S} \in \Sh(\Def, \Alg)$. 
For a $p$-adically flat ring spectrum $R$, we can form the  tensor product $L_{K(n)}(S \otimes R)$, and this is also a $K(n)$-local $E$-algebra. Since we have
\[
L_{K(n)}(S \otimes R) = S \cotimes[E] L_{K(n)}(E \otimes R)
\]
(where we take the tensor product $\cotimes[E]$ in the $K(n)$-local category), and the functors $\CAlg_{E(n)}^{K(n),\text{even}} \to \Alg_{\TT}$ and $\Alg_{\TT} \to \Sh(\Def, \Alg)$ are monoidal (\cite{Rezk09}), we get:
\begin{cor}
    Let $S$ be a $K(n)$-local $E$-algebra, and $R$ a $p$-adically flat ring spectrum. Let $\psi : \pi_0R \to \pi_0R$ be the lift of Frobenius associated to the $\delta$-ring structure of $\pi_0R$.
    
    The $\TT$-algebra structure on $\pi_0(S \cotimes R)$ corresponds to a sheaf $B \in \Sh(\Def, \Alg)$, such that it assigns to a deformation $G$ over $P$ the ring $B_{P}(G) = B^{\pi_0S}_P(G) \cotimes \pi_0R$, and for a morphism $f : G \to G'$ deforming $\Frob^r$:
    \[
    B_P(f) = B^{\pi_0S}_P(f) \otimes \psi^r
    \]
\end{cor}
An essential component of the proof is the algebro-geometric understanding of the $E$-cohomology rings of finite groups. In particular, we use the well-established result that $\Spf(E^0(B\Sigma_{p^n}))$ is closely related to classifying subgroups of the formal group over $E$ (cf.\ \cite{Strickland98}).

Further, for an abelian group $A$, we use the geometry of the formal scheme $\Spf(E^0(BA))$ and its closed subschemes, which classify homomorphisms from $A$ to the formal group.

Finally, we establish a geometric interpretation of the ring $\pi_0(R^{\tau A})$, where $(-)^{\tau A}$ denotes the proper Tate construction, introduced in \cref{subsection:proper-tate-intro}. This interpretation is developed in \cref{subsection:calculations}.

\subsection{Motivation} \label{subsection:motivation}

It is of interest to study the mapping space $\Mor(R,S)$ for $p$-adically flat ring spectra $R$ and $S$.
This space can be viewed as a homotopical analogue of morphisms between flat algebras in ordinary algebra, and understanding its structure plays an important role in spectral algebraic geometry.
In general, however, this space is complicated and its homotopy groups are mostly unknown. 
Nevertheless, there are some cases where explicit calculations have been made (see, for example, \cite{CNY24}).

One way to approach this problem is to use the chromatic convergence theorem on $S$, giving $S = \lim_n L_{E(n)}S$ (note that the chromatic convergence is classically formulated for finite ring spectra, but it is also true for $p$-adically flat spectra). Thus:
\[
\Hom(R, S) = \lim \Hom(R, L_{E(n)}S) 
\]
The space $\Hom(R, L_{E(n)} S)$ is related to the space $\Hom(R, L_{K(n)}S)$, as $L_{E(n)}S$ is equivalent to the pullback $L_{E(n-1)}S \times_{L_{E(n-1)}L_{K(n)}S} L_{K(n)}S$.

We have 
\[
L_{K(n)}S = \left( L_{K(n)}(S \otimes E(n)) \right)^{h\GG_n}
\]
Where $\GG_n$ is the Morava stabilizer group. Our next step is to find the $\GG_n$-equivariant maps in $\Hom(R, L_{K(n)}(S \otimes E(n)))$. Finally, we use the fact that this is the same as $E(n)$-algebra maps $\Hom_{E(n)}(L_{K(n)}(R \otimes E(n)), L_{K(n)}(S \otimes E(n)))$. 

In this project, we aim to explicitly determine $\pi_0(L_{K(n)}(R \otimes E(n)))$, which could be the first step in this ambitious program. 

\subsection{Conventions}

We follow the following convention:

\begin{enumerate}

    \item For two spectra $X, Y$, we will denote their smash product as $X \otimes Y$ and call it the tensor product of the spectra (this is also sometimes denoted by $X \wedge Y$).

    \item Throughout the paper we will use the notation $\cotimes$ to mean completed tensor product. That is, for two complete local rings $R,S$, $R \cotimes S$ is completed with respect to both maximal ideals. 
    For a complete local ring $R$ and another ring $A$, $R \cotimes A$ will be completed with respect to the maximal ideal of $R$. 

    \item In all tensor products involving $K(n)$-local rings, $\cotimes$ will mean the $K(n)$-local tensor product. That is, for two ring spectra $R,S$ in the $K(n)$-local category, $R \cotimes S = L_{K(n)}(R \otimes S)$, and for a $p$-adically flat ring spectrum $R$, $R \cotimes E$ denotes $L_{K(n)}(R \otimes E)$.

    \item $E$ will always denote Morava $E$-theory of height $n$ at a prime $p$, where $n$ and $p$ are fixed throughout the paper.

    \item Throughout, all ring spectra are assumed to be $\EE_{\infty}$. 
\end{enumerate}

\subsection{Organization \& Proof Outline}

In \cref{section:preliminaries} we recall some algebraic background on $\delta$-rings and formal groups. We also review facts about $p$-adically flat spectra and the structure of the category of $K(n)$-local $E$-modules. In particular, we recall the result of \cite{Rezk09} which characterizes the power operations in the category of $K(n)$-local $E$-modules.
\\ \\ 
In \cref{section:construction} we discuss the algebraic construction of the functor $\Alg_{\delta} \to \Alg_{\TT}$, in a way that is essentially oblivious to the topological origin of the functor.
\\ \\ 
In \cref{section:proper-tate} we recall some facts about the proper Tate construction, in particular that it is exact, lax-symmetric monoidal, and admits a lax natural transformation $\Delta_G : X \to (X^{\otimes G})^{\tau G}$ called the Tate-diagonal. 
We also define the generalized Frobenius map $R \to R^{\tau G}$, and relate it with the more classical Tate-valued Frobenius map $R \to R^{tC_p}$, where $R$ is a $p$-adically flat ring spectrum, and $G$ is an abelian $p$-group.
\\ \\
In \cref{section:correctness} we prove that the construction described in \cref{section:construction} is the missing functor from diagram~(\ref{diag:dotted-functor}). The argument proceeds in stages:
\begin{enumerate}
    \item[-] In \cref{subsection:recollection-on-T-algs}, we recall some facts about the monad $\TT$. 
    
    \item[-] In \cref{subsection:reduction-universal}, we construct two maps
    $\pi_0(R \cotimes E) \to \pi_0((R \cotimes E)^{B\Sigma_{p^m}}) / (I_{tr}) = (R \cotimes E)^0(B\Sigma_{p^m})/(I_{tr})$,
    where $I_{tr}$ denotes a certain transfer ideal. We show that it suffices to prove that these two maps agree. They encode the action of the $\TT$-algebra on the universal deformation: the expected action and the action coming from the $\TT$-algebra structure.
    \item[-] In \cref{subsection:abelian-subgroups} we use a certain injection 
    \[
    (R \cotimes E)^0(B\Sigma_{p^m})/(I_{tr}) \to \bigoplus_{A \subseteq \Sigma_{p^m}} (R \cotimes E)^0 (BA) / (I_{tr})
    \]
    where $A$ runs over abelian subgroups, to reduce our problem to showing the two maps agree after restriction to abelian subgroups. 

    \item[-] In \cref{subsection:comp-proper-tate}, we construct a map $(R \cotimes E)^0(BA)/(I_{tr}) \to \pi_0((R \cotimes E)^{\tau A})$, and analyze its basic properties. In \cref{subsection:calculations}, we show that this map is injective. Consequently, we may postcompose our two previously defined maps with it, reducing the proof to showing that the resulting maps $\pi_0(R \cotimes E) \to \pi_0((R \cotimes E)^{\tau A})$ coincide.

    The proof relies on the geometric interpretation of the rings and maps involved.

    \item[-] Finally, in \cref{subsection:putting-it-all-together} we use the fact that $\pi_0(R) \otimes \pi_0(E) \to \pi_0(R \cotimes E)$ is a surjection, to reduce our problem once again to checking that two maps $\pi_0(R) \otimes \pi_0(E) \to \pi_0((R \cotimes E)^{\tau A})$ agree. We then show the two maps do, in fact, agree, combining ideas from \cref{section:proper-tate} and \cref{section:correctness}.
\end{enumerate}

\subsection{Acknowledgments}

I would like to thank my advisor, Shachar Carmeli, for his guidance and support throughout this project; Akhil Mathew for sharing ideas related to this work; and Noam Nissan for her valuable insight and for taking the time to read and comment on drafts of this work. 

This project is partially supported by BSF grant 2024766.

\section{Preliminaries} \label{section:preliminaries}

In this section, we review background material. We recall basic properties of $\delta$-rings, formal groups, and deformation theory. We later introduce $p$-adically flat spectra, review Morava $E$-theory and modules over it, and summarize key facts about the proper Tate construction in higher algebra.

Throughout the whole paper we fix a prime $p$.

\subsection{$\delta$-Rings} \label{subsection:delta-rings}

\begin{definition}[{\cite[Definition 1]{Joyal85}}]
    A $\delta$-ring is a ring $R$, together with a map $\delta \colon R \to R$, such that $\delta(1)=0$,  $\delta$ satisfies the product rule
\[
\delta(ab) = \delta(a)b^p + a^p \delta(b) + p \delta(a) \delta(b)
\]
and the sum rule
\[
\delta(a+b) = \delta(a) + \delta(b) + \frac{a^p + b^p - (a+b)^p}{p}
\]
where the division by $p$ is formal. 
\end{definition}

Given a $\delta$-ring $R$, we define $\psi(x) := x^p + p\delta(x)$. It is straightforward to verify that $\psi$ is a ring endomorphism (unlike $\delta$ itself). Moreover, $\psi(a) \equiv a^p \pmod{pR}$, so $\psi$ reduces modulo $p$ to the Frobenius endomorphism $R \to R/p$. 

\begin{definition}
    A ring endomorphism $\psi\colon R \to R$ is called a lift of Frobenius if its reduction modulo $pR$ agrees with the Frobenius homomorphism $R \to R/p$.
\end{definition}

Thus, there is a forgetful functor from the category of $\delta$-rings to the category of rings equipped with a lift of Frobenius. If $R$ is $p$-torsion free, then $\delta$-ring structures on $R$ are in bijection with lifts of Frobenius on $R$ (see \cite{Joyal85}).

In fact, for $p$-torsion-free rings, the axioms defining a $\delta$-ring are precisely those ensuring that the map $x \mapsto x^p + p\delta(x)$ is a ring endomorphism. We denote the category of $\delta$-rings by $\Alg_{\delta}$, and torsion-free $\delta$-rings by $\Alg_{\delta,\text{tf}}$.

\subsection{Formal Groups and Deformations} \label{subsection:formal-groups}

We fix a perfect field $\kappa$ of characteristic $p$ and a formal group $\Gamma$ over $\kappa$ of height $n$. This means that $\Gamma$ has underlying formal scheme $\Spf \kappa[\![x]\!]$, together with a group structure $\Gamma \times \Gamma \to \Gamma$ given in coordinates by a power series $f(x,y) \in \kappa[\![x,y]\!]$ satisfying

\[
f(x,0)=f(0,x)=x \;\;,\;\; f(x,y)=f(y,x) \;\;,\;\;f(x,f(y,z))=f(f(x,y),z)
\]
For $m \ge 0$, we define the $m$-th series of $f$ recursively, by $[1](x)=x$ and $[m](x) = f(x, [m-1](x))$. 
Then, as we chose our formal group to be of height $n$, $[p](x)$ has the form $[p](x) = \lambda x^{p^n} + (\text{higher terms})$, where $\lambda \ne 0$.

\begin{definition}[Homomrophism, Isogeny]
    Given formal groups $\Gamma$ and $\Gamma'$, a homomorphism $g \colon \Gamma \to \Gamma'$ is an a homomorphism of schemes preserving the group structure (i.e., $m_{\Gamma'} \circ (g \times g)=g \circ m_{\Gamma}$, where $m_{\Gamma} \colon \Gamma \times \Gamma \to \Gamma$ is the multiplication map). 
    Note that $g$ is an isomorphism if and only if it is an isomorphism of formal schemes.  An isogeny is a homomorphism $g \colon \Gamma \to \Gamma'$ of formal groups, for which $g^*\colon \scO_{\Gamma'} \to \scO_{\Gamma}$ is finite and free.
\end{definition}
\begin{definition}[Deformation of Formal Group] 
    Let $\Gamma$ be a formal group over the field $\kappa$.
    A deformation of $\Gamma$ to a complete local ring $R$ is a triple $(G, i, \alpha)$, where $G$ is a formal group over $R$, $i \colon \kappa \to R/\frakm_R$ is a field embedding, and $\alpha \colon p^*G \xrightarrow[]{\cong} i^* \Gamma$ is an isomorphism, where $p \colon R \to R/\frakm_R$ is the quotient map.
\end{definition}
\noindent Here $f^*G$ denotes the pullback of the formal scheme $G$ along the morphism $f$.
\\\\
Suppose that $R$ has characteristic $p$, so that the absolute Frobenius $\phi(x)=x^p$ is an endomorphism of $R$. Forming the pullback $\phi^*G$, we obtain the relative Frobenius isogeny $\Frob\colon G \to \phi^*G$. By applying this construction repeatedly, we can form the iterated relative Frobenius map $\Frob^r \colon G \to (\phi^r)^* G$.
% If we assume that $\phi$ is injective, then $\phi^*G$ has a formal group law equal to that of $G$, with coeficients raised to the power of $p$. $Frob \colon G -> \phi^*G$ corresponds to $t \mapsto t^p$.  
\begin{definition}[Defomration of Frobenius]
    Let $(G,i,\alpha)$ and $(G',i',\alpha')$ be deformations of $\Gamma$ to $R$, and let $f\colon G \to G'$ be a homomorphism.
    We say that $f$ is a deformation of $\Frob^r$ if $i \circ \phi^r = i'$ and the following diagram commutes:
\[
\begin{tikzcd}
i^* \Gamma  \arrow[r, "\Frob^r"] \arrow[d, "\alpha"] & (\phi^r)^* i^* \Gamma \arrow[d, "\alpha'"] \\
p^*G \arrow[r, "f"]                                  & p^* G'                  
\end{tikzcd}
\]
\end{definition}

\begin{rmrk}
    When $r = 0$, the above definition recovers the standard definition of a homomorphism of deformations. 
\end{rmrk}

Let $\Def_R$ denote the category of deformations of $\Gamma$ to $R$, whose morphisms are deformations of iterated Frobenius maps. Note that this category is graded, in the sense that the hom-sets have a grading according to which power of $\Frob$ is being deformed (and this is part of the data of a morphism). 

The assignment $R \mapsto \Def_R^{\op}$ defines a functor ${\scR}^{\op} \to \Cat$ (and in fact it is a sheaf with respect to some appropriate topology), where $\scR$ is the full subcategory of complete local rings, consisting of complete local rings with residue field of characteristic $p$.

We also consider the functor $R \mapsto \Alg_R$, and define $\Sh(\Def,\Alg)$ to be the category of natural transformations between these functors.
An object $B$ of this category assigns to each ring $R$ a functor $B_R\colon \Def_R^{\op} \to \Alg_R$, together with natural isomorphisms $B_f \colon f^*B_R \to B_{R'}f^*$ for each morphism $f\colon R \to R'$, satisfying suitable coherence conditions.

We refer to $\Sh(\Def,\Alg)$ as the category of sheaves on the deformation category.

\begin{rmrk} 
The object $\Def$ is sometimes called a heap, analogous to a stack but valued in categories rather than groupoids. The category $\Sh(\Def,\Alg)$ may thus be viewed as the category of sheaves of algebras on this heap. 
\end{rmrk}

\subsection{$p$-adically Flat Spectra} \label{subsection:flatness}

Let $\pcs$ denote the $p$-complete sphere spectrum. In this subsection, we introduce the notion of $p$-adic flatness over $\pcs$ and record some of its basic properties.

\begin{lem} \label{lem:flat-equiv}
    Let $X$ be a $p$-complete connective spectrum. The following conditions are equivalent:
    \begin{enumerate}
        \item[(i)]   $X \otimes \FF_p$ is discrete.

        \item[(ii)] $X$ is equivalent to a $p$-completed direct sum of copies of $\pcs$, as an $\pcs$-module.
        
    \end{enumerate}
\end{lem}

\begin{proof}
    For (i)$\Rightarrow$(ii), assume that $X \otimes \FF_p$ is discrete. Then it is an ordinary $\FF_p$-vector space, and hence of the form $X \otimes \FF_p \simeq \bigoplus_{a\in A} \FF_p$.
    By the Hurewicz theorem, for each $a\in A$ there exists a map $f_a\colon \pcs \to X$ whose image in $\FF_p$-homology represents the corresponding summand.
    These maps assemble to a morphism $f\colon (\bigoplus_{a\in A} \pcs)_p^{\wedge} \to X$, which induces an isomorphism on $\FF_p$-homology. By the Hurewicz theorem, $\FF_p$ detects isomorphisms for $p$-complete spectra, thus $f$ is an isomorphism.

    The implication (ii)$\Rightarrow$(i) is immediate: if $X \simeq (\bigoplus_{a\in A} \pcs)_p^{\wedge}$, then $X \otimes \FF_p \simeq \bigoplus_{a\in A} \FF_p$.
    
\end{proof}

\begin{definition} 
A connective, $p$-complete ring spectrum $X$ is called $p$-adically flat if it satisfies the equivalent conditions of \cref{lem:flat-equiv}. 
\end{definition}

Note that if $X$ is $p$-adically flat, than $\pi_0(X)$ is $p$-adically flat over $\ZZ_p$. Moreover, we will frequently use the fact that for a $p$-adically flat spectrum $X$, and a $p$-complete spectrum $F$, one has $\pi_0(X \cotimes F) = \pi_0X \cotimes_{\ZZ_p} \pi_0 F$, where the tensor product(s) are $p$-completed. Both of these facts follow immediately from expressing $X$ as a $p$-completed direct sum of copies of $\pcs$. 

Further:

\begin{thm}
    Let $X$ be a $p$-adically flat spectrum. Then the trivial map  $\text{triv}:X \to X^{tC_p}$ is an equivalence, where $X^{t C_p}$ is the Tate-construction.
\end{thm}

We will in fact need a generalization of this result, which will be given in \cref{thm:tate-is-equiv}. Finally, we are ready to introduce the $\delta$-ring structure on $\pi_0 R$ when $R$ is $p$-adically flat ring spectrum:

\begin{construction}[{\cite[Definition IV.1.1]{NS18}}]
    Suppose $R$ is a $p$-adically flat ring spectrum. 
    There exists a natural ring map $\phi \colon R \xrightarrow[]{\Delta} \left( 
    R^{\otimes p}\right)^{tC_p} \xrightarrow[]{m} R^{tC_p} \cong R$ where the first map is the Tate diagonal map (see \cite[Definition III.1.4]{NS18}, and \cref{cor:tate-diagonal}), the second map is multiplication, and the equivalence $R^{tC_p} \simeq R$ is induced by the trivial map $R \to R^{t C_p}$. 

    After applying $\pi_0$, $\pi_0 \phi \colon \pi_0 R \to \pi_0 R$ becomes a lift of Frobenius. As $\pi_0R$ is torsion free, this lift of Frobenius uniquely corresponds to a $\delta$-ring structure on $\pi_0R$.
\end{construction}

\subsection{Morava E-theory}
\label{subsection:morava-e-theory}
Let $E:=E(n)$ denote Morava $E$-theory associated to our fixed formal group $\Gamma$ over $\kappa$. $E$ satisfies $\pi_*E=LT_n[u^\pm]$, where $LT_n$ is the universal ring of deformations of $\Gamma$. Explicitly, deformations of $\Gamma$ to a ring $R$ corresponds uniquely to a ring maps $LT_n \to R$. Let $K(n)$ denote Morava K-theory, which is a ring spectrum satisfying $\pi_*K(n)=\FF_p[v_n^{\pm1}]$ where $|v_n|=2(p^n-1)$. Let $\CAlg^{K(n),\text{ev}}_{E(n)}$ denote the category of even $K(n)$-local $E(n)$-algebras. Let $L_{K(n)}$ denote the localization functor on spectra with respect to $K(n)$. 
\\ \\ 
If $M$ is a $K(n)$-local $E(n)$-module, we may form the $m$-fold tensor product $M^{\otimes_E m}$. This module is acted on by $\Sigma_m$, the group of permutations of $m$ elements, via permutations of the factors. We define $\PP_m(M) := \left( M^{\otimes_E m}\right)_{h\Sigma_m}$, the homotopy orbits. Further define $\PP(M) := \bigoplus_{m \ge 0} \PP_m(M)$. $\PP$ is a monad on the category of $E(n)$-modules, and any commutative $E(n)$-algebra defines an algebra for this monad, via the multiplication map. 

A non-trivial fact about $\PP_m$ is that if $M$ is a free and finitely generated $E(n)$-module, then $L_{K(n)}\PP_m(M)$ is also finite free (see \cite{Rezk09}).

For finite free $E(n)$-modules, there is an equivalence of categories $\pi_*:\ho\left(\Mod_{E(n)}^{\text{ff}}\right) \to \Mod_{E(n)_*}^{\text{ff}}$, where ff means finite and free $E$-algebras (see \cite{Rezk09}). 
This allows us to define a functor $\TT_m$ on $\Mod^{\text{ff}}_{E(n)_*}$, the composition of $\pi_*$ with $L_{K(n)}\PP_m$, and the inverse of $\pi_*$.
We extend $\TT_m$ to all $E(n)_*$ modules via left Kan extension along the inclusion of finite free modules in the category of all modules. Setting $\TT := \bigoplus_{m \ge 0} \TT_m$, we obtain a monad on the category of $E_*$-modules. If $R$ is a $K(n)$-local $E$-algebra, then $\pi_*(R)$ admits the structure of a $\TT$-algebra. 

There is a way to construct a ring $\Gamma$ such that $\TT$-algebras are naturally $\Gamma$-algebras, this identifies 
$\TT$-algebras with $\Gamma$-algebras satisfying a certain congruence condition. 
Concretely, for a graded $\Gamma$-algebra $B$, this condition requires $ \sigma x \equiv x^p \pmod{pB}$ for all $x \in B_0$, where $\sigma \in \Gamma$ is some special element. (See \cite{Rezk09} for a more detailed discussion of the construction of $\TT$ and $\Gamma$).

In \cite{Rezk09}, Rezk also shows an equivalence of categories: $\Sh(\Def, \Alg) \approx \Alg_{\Gamma}$. Moreover, he characterizes which objects in $\Sh(\Def, \Alg)$ correspond to $p$-torsion free $\Gamma$-algebras satisfying the congruence condition. For the sake of completeness, we will describe this condition now, even though it will not be relevant to our results.
\\ \\ 
Suppose $R$ has characteristic $p$. The absolute Frobenius $\phi$ defines a base-change functor $\phi^* \colon \Alg_R \to \Alg_R$ ($\phi^*(A) = A \otimes_{\phi} R$). There is a natural transformation $\Frob: \phi^* \Rightarrow \Id$, defined such that the composition of the map $A \to \phi^*A$ with $\Frob_A \colon \phi^*A \to A$ yields the absolute Frobenius on $A$ (i.e., the map $A \otimes_\phi R \to A$ maps $a \otimes r \mapsto r a^p$). 

An object $B \in \Sh(\Def, \Alg)$ satisfies the the \emph{Frobenius congruence condition} if for all $R$ of characteristic $p$, and $G \in \Def_R$, the following diagram commutes:
\[
\begin{tikzcd}[row sep=large, column sep = large]
    \phi^* B_R(G) \arrow[r, "B_\phi", "\simeq"'] \arrow[rd, "\Frob"'] & B_R(\phi^* G) \arrow[d, "B_R(\Frob)"] \\ 
    & B_R(G)
\end{tikzcd}
\]
Informally, this means that, up to the equivalence $B_\phi$, $B$ carries the Frobenius isogeny of formal groups to  the relative Frobenius of algebras.

\begin{rmrk} 
    We have intentionally suppressed some of the discussion relating to the treatment of the even and odd part of $\pi_*R$ where $R$ is a $K(n)$-local $E(n)$-algebra. To treat not necessarily even spectra we need to enlarge our algebraic categories (both $\Sh(\Def, \Alg)$ and $\Alg_\TT$) by allowing them to have, in addition to their even part, an odd part. Throughout the paper, we will only treat even $E$-algebras. 
\end{rmrk}

As we are interested in the associated sheaf of $R \cotimes E$ where $R$ is a $p$-adically flat ring spectrum, it will be useful to find the underlying $E_0$-module:

\begin{prop}
    If $X$ is a $p$-adically flat spectrum, then the map
    \[
    \pi_0 X \otimes \pi_0 E \to \pi_0(X \otimes E) \to  \pi_0(X \cotimes E) = \pi_0(L_{K(n)}(X \otimes E)) 
    \]
    Is a completion of $\pi_0X \otimes \pi_0E$ with respect to the maximal ideal of $\pi_0E=E_0$. I.e., \[
    \pi_0(X \cotimes E) = \pi_0(L_{K(n)}(X \otimes E)) = \pi_0 X \cotimes \pi_0 E
    \]
\end{prop}
\begin{proof}
    Let $X = (\bigoplus_{a \in A} \pcs)_p^\wedge $. $K(n)$-localization factors through $p$-completion, thus:
    \[
    X \cotimes E = L_{K(n)}(X \otimes E) = L_{K(n)}\left( \bigoplus_{a \in A} E \right)
    \]
    By \cite[Proposition 8.4]{HS99}, $\pi_0(L_{K(n)}(\bigoplus_{a \in A}E))$ is obtained from $\pi_0(\bigoplus_{a \in A}E) = \bigoplus_{a \in A} E_0$ by completion with respect to the maximal ideal of $E_0$. This is exactly the completed tensor product $\pi_0 X \cotimes \pi_0E$.
\end{proof}

\begin{cor} \label{cor:pi-0-of-flat-tensor-E}
    Let $R$ be a $p$-adically flat ring spectrum. Then, as $E_0$-algebras:
    \[
    \pi_0(R) \cotimes E_0 = \pi_0(R \cotimes E)
    \]
\end{cor}
\begin{proof}
    We know this is true as $E_0$-modules, thus we only need to verify that that the map $\pi_0(R) \cotimes \pi_0(E) \to \pi_0(R \cotimes E)$ is a ring map.
    
    Since the composition
    \[
    \pi_0(R) \otimes \pi_0(E) \to \pi_0(R \otimes E) \to \pi_0(R \cotimes E)
    \]
    is a ring map, by completing the source with respect to the maximal ideal of $E_0$ we get that the original map is also a ring map. 
\end{proof}

\subsection{Proper Tate Construction} \label{subsection:proper-tate-intro}

We have already discussed the ordinary Tate construction $(-)^{tC_p}$ in \cref{subsection:flatness}; here we introduce a useful generalization.

Informally, for a spectrum $X$ with a $G$-action, where $G$ is a finite group, the proper Tate construction of $X$, $X^{\tau G}$, captures those fixed-point phenomena that do not arise from transfer maps along proper orbits $G/H$.
This construction is functorial $(-)^{\tau G} : \Sp^{BG} \to \Sp$, and in fact lax symmetric monoidal; see \cref{prop:proper-tate-is-lax}. 

For completeness, we recall one model for the proper Tate construction.

\begin{definition} \label{def:proper-tate}
    Let $X \in \Sp^{BG}$ be a spectrum with a $G$-action, where $G$ is some finite group. Define $X^{\tau G}$ to be the cofiber of the canonical map
    \[
    \colim_{G/H \in \Orb_G} X^{hH} \to X^{hG}
    \]
   Where $\Orb_G$ denotes the orbit category of $G$, whose objects are the orbits $G/H$ and whose morphisms are $G$-equivariant maps $G/H \to G/K$; and the maps $X^{hH} \to X^{hG}$ are relative transfer maps.
\end{definition}
If the action of $G$ on $X$ is trivial and $H \subseteq G$ is a subgroup, there is a natural map $\can_H^G \colon X^{\tau H} \to X^{\tau G}$, induced by the inclusion $X^{hH} \to X^{hG}$. 
Indeed, we have the composition $X^{hH} \to X^{hG} \to X^{\tau G}$, and the pre-composition of this map with any transfer is $0$, thus it induces a map $X^{\tau H} \to X^{\tau G}$. 

If $R$ is a ring spectrum, one also obtains a generalized Frobenius map
\[
\phi^G \colon R \xrightarrow[]{\Delta_G} \left(R^{\otimes G} \right)^{\tau G} \to R^{\tau G}
\]
(See \cref{cor:tate-diagonal}).
There is, in fact, a relative version of the Frobenius $\phi_{H}^G \colon R^{\tau H} \to R^{\tau G}$, which is defined similarly. 

In the special case $G = C_p$, the proper Tate construction agrees with the usual Tate construction, $(-)^{\tau C_p} = (-)^{tC_p}$, and $\phi^{C_p} = \phi$ recovers the classical Tate-valued Frobenius.
\\ \\ 
We recall the following result for $p$-adically flat spectra.

\begin{thm} \label{thm:tate-is-equiv}
    Let $X$ be a $p$-adically flat spectrum, and $G$ be a finite $p$-group. Equip $X$ with the trivial $G$-action. Then the trivial map $X \to X^{\tau G}$ is an equivalence.
\end{thm}

\begin{proof}
    In \cite[Corollary 6.7.1]{Yuan21} this is proven when $G$ is an $\FF_p$ vector space. The same argument applies to any finite $p$-group, since the only property used is that $BG$ admits a cellular decomposition with finitely many cells in each degree.

    % In \cite{Carl84} it shown for finite spectra, in particular for the (p-complete) sphere spectrum $\pcs$. 

    % Since $X$ is flat, we can write $X = \left( \bigoplus_{a \in A} \pcs\right)_{p}^{\wedge}$. Since the action is trivial, $X^{hG} = X^{BG}$:

    % \[
    % X^{hG} = X^{BG} = \left( \bigoplus_{a \in A} \pcs\right)
    % \]

    % Note that we have $X^{hG} = X \cotimes (\pcs)^{hG}$ (with respect to the trivial action)\shachar{this is not that formal.}. Further, the transfer $X^{hH} \to X^{hG}$ is just the base change of the transfer $\SSS^{hH} \to \SSS^{hG}$ to $R$.
    
    % Hence, using the fact that base-change (i.e., tensor) commutes with colimits, the map:
    % \[
    % \colim_{G/H \in \text{Orb}_{G}} R^{hH} \to R^{hG}
    % \]
    % Is just the base change of the map:
    % \[
    % \colim_{G/H \in \text{Orb}_{G}} \SSS^{hH} \to \SSS^{hG}
    % \]
    % to $R$. The cofiber of this map is $\SSS^{\tau G} = \SSS$, and since $R$ is flat and $R \otimes (-)$ commutes with taking cofiber, we get $R^{\tau G} = R \otimes \SSS^{\tau G} = R$. 
\end{proof}

Let $\scQ$ be the symmetric monoidal span category $\text{Span}(\text{Finite Groups}, \text{surj}, \text{inj})$. Explicitly, $\scQ$ is a 1-category whose objects are finite groups and whose morphisms are spans of the form
\[\begin{tikzcd}[column sep = small]
	& G \\
	H && K
	\arrow[two heads, from=1-2, to=2-1]
	\arrow[hook, from=1-2, to=2-3]
\end{tikzcd}\]
where the left leg $G \twoheadrightarrow H$ is surjective and the right leg $G \hookrightarrow K$ is injective.

\begin{thm}{{\cite[Theorem 3.15]{Yuan21}}} \label{thm:proper-tate-functoriality}
    There exists an oplax functor $\tau \colon \scQ \to \Fun(\CAlg, \CAlg)$. 
    On objects, $\tau$ sends a finite group $G$ to the functor $(-)^{\tau G}$, taken with respect to the trivial action.

    A surjection $H \twoheadleftarrow G$ is sent to the canonical map $\can_H^G \colon (-)^{\tau H} \Rightarrow (-)^{\tau G}$. An injection $H \hookrightarrow G$ is sent to the relative Frobenius map $\phi_H^G \colon (-)^{\tau H} \Rightarrow (-)^{\tau G}$.
\end{thm}

In the case of complex-oriented ring spectra, the proper Tate construction admits a simple algebraic description:

\begin{thm} [{\cite[Proposition 5.10]{AMR23}}] 
    \label{thm:tate-of-complex-oriented}
    Let $R$ be a complex-oriented ring spectrum and let $G$ be a finite group. Then the canonical map $R^{hG} \to R^{\tau G}$ exhibits $R^{\tau G}$ as the localization of $R^{hG}$ with respect to the Euler class of the reduced complex regular representation of $G$. That is, if $e = e(\widetilde{\rho _G})$, then $R^{\tau G} = R^{hG}[e^{-1}]$. 
\end{thm}

\section{Construction of the Functor $\Alg_{\delta} \to \Sh(\Def, \Alg)$}
\label{section:construction}

In this section, we construct our functor $\Alg_{\delta, \text{tf}} \to \Alg_{\TT}$. Recall that there is a functor $\Alg_{\TT} \to \Sh(\Def, \Alg)$, which, when restricted to $p$-torsion free algebras, is fully faithful and has a well-understood essential image. 
Specifically, we construct a functor $\Alg_{\delta} \to \Sh(\Def, \Alg)$, and show that it lands in the subcategory of $\TT$-algebras for torsion free $\delta$-rings. This will be good enough for us, as in the case of a $p$-adically flat ring spectrum $R$, $\pi_0 R$  is $p$-torsion free (as it is a $p$-adically flat $\ZZ_p$-module). 

First, it will be useful for us to construct a functor from the category of rings equipped with an endomorphism $\Alg_{\psi}$ to the category $\Sh(\Def,\Alg)$.

\begin{construction} \label{const:sheaf-on-def}
    Let $A$ be a ring equipped with an endomorphism $\psi \colon A \to A$. We define an object $B^A \in \Sh(\Def, \Alg)$. For a ring $R \in \scR$ (i.e., $R$ is a local ring and $R/m_{R}$ has characteristic $p$), and a deformation $(G, i, \alpha)$ over $R$, we define $B^A_R(G) :=  A \cotimes R$, where the tensor product is completed with respect to the maximal ideal of $R$.
    
    For a morphism $f \colon G \to G'$ which deforms $\Frob^r$, we define $B^A_R(f) := \psi^r \cotimes \Id_R$.

    For a map $f \colon R \to R'$, we must give a natural isomorphism $B^A_f \colon f^* B_R^A \to B_{R'}^A f^*$. For a deformation $(G, i, \alpha)$ over $R$, we have 
    \[
    (f^* B_R^A)(G) = f^*(A \cotimes R) = (A \cotimes R) \cotimes[R] R'
    \]
    We also have, $(B_{R'}^A f^*)(G) = A\cotimes R'$. These are naturally isomorphic, as \[
    (A \cotimes R) \cotimes[R] R' \cong A \cotimes (R \cotimes[R] R') \cong A \cotimes R'
    \]

\end{construction}

\begin{prop}
    $B^A_{R} \colon \Def_{R}^{\op} \to \Alg_R$ is functorial.
\end{prop}

\begin{proof}
Let us verify that $B^A_R \colon \Def_R^{\op} \to \Alg_R$ is a well defined functor. First, $B^A(\Id_G) = \psi^0 \cotimes \Id_R = \Id_{A \cotimes R}$. 

Next, we need to check that $B^A_R(f' \circ f) = B^A_R(f') \circ B_R^A(f)$ for $f \colon G \to G'$ and $f' \colon G' \to G''$. If $f$ deforms $\Frob^r$ and $f'$ deforms $\Frob^{r'}$, then $i' = i\circ \phi^{r}$ and $i'' = i'\circ \phi^{r'}$, so $i'' = i \circ \phi^{r + r'}$. Moreover, we have the following commutative diagram:

\[
\begin{tikzcd}
i^* \Gamma  \arrow[r, "\Frob^r"] \arrow[d, "\alpha"] & (\phi^r)^* i^* \Gamma \arrow[d, "\alpha' "] \arrow[r, "\Frob^{r'}"] & (\phi^{r + r'})^* i^* \Gamma \arrow[d, "\alpha'' "] \\
p^*G \arrow[r, "f"] & (p')^* G' \arrow[r, "f'"] & (p'')^* G
\end{tikzcd}
\]

The composition of the top two arrows is exactly $\Frob^{r+r'}$; thus $f' \circ f$ deforms $\Frob^{r+r'}$. Hence 
\[
B^A_R(f' \circ f) = \psi^{r+r'} \cotimes \Id_R = (\psi^{r'} \cotimes \Id_R) \circ (\psi^{r} \cotimes \Id_R) =
B^A_R(f') \circ B_R^A(f)
\]
\end{proof}

\begin{prop}
    The sheaf $B^A$, constructed \cref{const:sheaf-on-def}, satisfies the requisite coherence conditions. 
\end{prop}
\begin{proof} 
First, note that $B^A_{\Id_R} \colon B^A_R \to B^A_R$ is the isomorphism $A \cotimes R \cotimes R \to A \cotimes R$ which is the identity under the identification $R \cotimes R \simeq R$.

Further, suppose that we have $f \colon R \to R'$ and $g \colon R' \to R''$. 
We need to verify that 
\[
B^A_{g \circ f} = B^A_g f^* \circ g^* B^A_f
\]
For a given deformation over $R$, the first transformation gives the isomorphism $ (A \cotimes R) \cotimes[R] R'' \to A \cotimes R''$. The latter gives the composition:
\[
((A \cotimes R) \cotimes[R] R') \cotimes[R'] R'' \xrightarrow[]{\sim} (A \cotimes R') \cotimes[R'] R'' \xrightarrow[]{\sim} A \cotimes R''
\]
Under the identification $(gf)^*(R \cotimes A) \cong g^* f^* (R \cotimes A)$, both maps agree.
\end{proof}

So far we have only defined the action of the functor on objects. Let us now describe the functor on morphisms:

\begin{construction} \label{const:on-maps}
    Let $F \colon A \to A'$ be a map of rings equipped with an endomorphism, i.e., $F \circ \psi_A = \psi_{A'} \circ F$. We construct a transformation $B^A \Rightarrow B^{A'}$.

    For a deformation $G$ over $R$, it gives the map $B^A_R(G) = A \cotimes R \xrightarrow[]{F\cotimes \Id_R} A' \cotimes R$. For a deformation $g \colon G \to G'$ of $\Frob^m$, the following diagram commutes:
    \[\begin{tikzcd}[column sep=large]
	{B^A_R(G) = A \cotimes R} & {B^{A'}_R(G) = A' \cotimes R} \\
	{B^A_R(G) = A \cotimes R} & {B^{A'}_R(G) = A' \cotimes R}
	\arrow["{F \cotimes \Id_R}", from=1-1, to=1-2]
	\arrow["{\psi^r \cotimes \Id_R}"', from=1-1, to=2-1]
	\arrow["{\psi^r \cotimes \Id_R}", from=1-2, to=2-2]
	\arrow["{F \cotimes \Id_R}"', from=2-1, to=2-2]
\end{tikzcd}\]
    Thus this transformation really is natural.
\end{construction}

\begin{claim}
    The transformation defined above $B^A \Rightarrow B^{A'}$ is compatible with the coherence of $B^A$ and $B^{A'}$.
\end{claim}
\begin{proof}
    Let $f \colon R \to R'$ be a map of rings in $\scR$, and we need to show that the following diagram commutes

\[\begin{tikzcd}
	{f^*B^A_R} & {B^A_{R'}f^*} \\
	{f^*B^{A'}_R} & {B^{A'}_{R'}f^*}
	\arrow["\sim", from=1-1, to=1-2]
	\arrow[from=1-1, to=2-1]
	\arrow[from=1-2, to=2-2]
	\arrow["\sim", from=2-1, to=2-2]
\end{tikzcd}\]

Indeed, it is enough to see this for objects, which is immediate, as for a deformation $G$ over $R$, the following diagram commutes

\[\begin{tikzcd}[column sep=large]
	{(A \cotimes R) \cotimes[R] R'} & {A \cotimes R'} \\
	{(A' \cotimes R) \cotimes[R] R'} & {A \cotimes R'}
	\arrow["\sim", from=1-1, to=1-2]
	\arrow["{(F \cotimes \Id_R) \cotimes \Id_{R'}}"', from=1-1, to=2-1]
	\arrow["{F \cotimes \Id_{R'}}", from=1-2, to=2-2]
	\arrow["\sim", from=2-1, to=2-2]
\end{tikzcd}\]

\end{proof}

\begin{rmrk} \label{rmrk:maps-of-gamma-rings}
    By coherence, the transformation $B^A \Rightarrow B^{A'}$ is determined uniquely by the map of the underlying rings, $B^A_{E_0}(\GG) \to B^A_{E_0}(\GG)$, where $\GG$ is the universal deformation over $E_0$. This is just a reflection of the fact that a map of $\Gamma$-algebras is a map of the underlying rings. In our case, the map $A \to A'$ is mapped to its base change $A \cotimes E_0 \to A' \cotimes E_0$.
\end{rmrk}

\noindent Thus we have a functor $\Alg_{\psi} \to \Sh(\Def, \Alg)$. We can compose this functor with the functor $\Alg_{\delta} \to \Alg_{\psi}$ which associates to a $\delta$-ring the lift of Frobenius $\psi(x) = x^p + p\delta(x)$. This gives us our desired functor $\Alg_{\delta} \to \Sh(\Def, \Alg)$.

Later in the paper we will show the functor makes diagram (\ref{diag:dotted-functor}) commute (more precisely, if we compose the missing functor with the functor $\Alg_{\TT} \to \Sh(\Def, \Alg)$ we will get the functor we have just described). From this it follows that the resulting sheaf $B^A$ must satisfy the Frobenius congruence condition. Nonetheless, it is easy to check explicitly that the condition is satisfied.

\begin{prop}
    The sheaf $B^A$, constructed in \cref{const:sheaf-on-def}, satisfies the congruence condition, provided $\psi:A\to A$ is a lift of Frobenius. 
\end{prop}
\begin{proof}

Suppose that $\psi \colon A \to A$ is a lift of the Frobenius endomorphism, i.e., $\psi(x) \equiv x^p \mod{pA}$. Suppose $R$ is a ring of characteristic $p$, and let $\phi \colon R \to R$ be the absolute Frobenius endomorphism. Let $G \in \Def_R$. Then:
\[
\phi^*B^A_R(G) = \phi^*(A \cotimes R) = (A \cotimes R) \cotimes[R] R_0
\]
Where $R_0 = R$ denotes the $R$-module structure given by $\phi$. 

We need to check the commutativity of the following diagram:
\[
\begin{tikzcd}[row sep=large, column sep = large]
   (A \cotimes R) \cotimes[R] R_0   \arrow[r , "\sim"'] \arrow[rd, "\Frob"'] & A \cotimes R \arrow[d, "B_R(\Frob) = \psi \cotimes \Id_R"] \\ 
    & A \cotimes R
\end{tikzcd}
\]
The top morphism is given by  $a \otimes r \otimes r_0 \mapsto a \otimes (r_0 r^p)$. Hence the composition of this map with $\psi \cotimes \Id_R$ is given by $a \otimes r \otimes r_0 \mapsto \psi(a) \otimes (r_0 r^p)$.

The relative Frobenius map is given by $a \otimes r \otimes r_0 \mapsto r_0(a \otimes r)^p$. Since $R \otimes pA = 0$, as $R$ has characteristic $p$, we have:
\[
r_0 (a \otimes r)^p = r_0 (a^p \otimes r^p) = r_0(\psi(a) \otimes r^p) = \psi(a) \otimes (r_0 r^p) 
\]
And hence the diagram commutes.

\end{proof}

\begin{cor}
    The construction $A \mapsto B^A$, along with the action on maps described in \cref{const:on-maps}, defines a functor $\Alg_{\delta} \to \Sh(\Def, \Alg)$. Restricted to $p$-torsion free rings, this functor lands in the full subcategory of $\TT$-algebras. 
\end{cor}

\section{Proper Tate Construction}
\label{section:proper-tate}

In this section we establish several useful facts about the proper Tate construction, defined in \cref{def:proper-tate}. 
In \cref{subsection:functorial-properties} we expand on the functoriality of the construction. In \cref{subsection:tate-diagonal} we introduce the useful Tate diagonal map. Finally, in \cref{subsection:gen-frob} we specialize to $p$-adically flat ring spectra, where the Tate diagonals interact well with one another.

\subsection{Exactness \& Multiplicativity} \label{subsection:functorial-properties}

Fix some finite group $G$. 

\begin{prop} \label{prop:tate-is-exact}
    The proper Tate construction $(-)^{\tau G}:\Sp^{BG} \to \Sp$ is exact.
\end{prop}
\begin{proof}
   Suppose $X \to Y \to Z$ is a (co)fiber sequence of spectra with $G$-action. Then, since transfer maps are natural, we have the following commutative diagram:
   \[\begin{tikzcd}
	{\colim_{G/H \in \Orb_{G}} X^{hH}} & {\colim_{G/H \in \Orb_{G}} Y^{hH}} & {\colim_{G/H \in \Orb_{G}} Z^{hH}} \\
	{X^{hG}} & {Y^{hG}} & {Z^{hG}} \\
	{X^{\tau G}} & {Y^{\tau G}} & {Z^{\tau G}}
	\arrow[from=1-1, to=1-2]
	\arrow[from=1-1, to=2-1]
	\arrow[from=1-2, to=1-3]
	\arrow[from=1-2, to=2-2]
	\arrow[from=1-3, to=2-3]
	\arrow[from=2-1, to=2-2]
	\arrow[from=2-1, to=3-1]
	\arrow[from=2-2, to=2-3]
	\arrow[from=2-2, to=3-2]
	\arrow[from=2-3, to=3-3]
	\arrow[from=3-1, to=3-2]
	\arrow[from=3-2, to=3-3]
\end{tikzcd}\]

    The middle row is a cofiber sequence since $(-)^{hG}$ is exact. 
    
    Let us explain why the top row is a cofiber sequence. $X^{h(-)}\colon G/H \mapsto X^{hH}$ is a functor $\Orb_G \to \Sp$, and similarly for $Y$ and $Z$. Since the relative transfer is natural, there are natural transformations $ X^{h(-)} \Rightarrow Y^{h(-)}$ and $Y^{h(-)} \Rightarrow Z^{h(-)}$. By \cite[Proposition 1.1.3.1]{Lur17}, the category $\Fun(\Orb_{G}, \Sp)$ is stable and cofibers are computed objectwise. Since each $X^{hH} \to Y^{hH} \to Z^{hH}$ is a cofiber sequence, so is:
    \[
    X^{h(-)} \Longrightarrow Y^{h(-)} \Longrightarrow Z^{h(-)}
    \]
    Finally, since $\colim \colon \Fun(\Orb_G, \Sp) \to \Sp$ is exact, it follows that the top row is a cofiber sequence.
    
    Hence the bottom row is a cofiber sequence (this is the nine-lemma for stable infinity categories). 
\end{proof}

Recall that the functor $(-)^{hG} \colon \Sp^{BG} \to \Sp$ is lax symmetric monoidal; the functor is right adjoint to the monoidal trivial action functor $\Sp \to \Sp^{BG}$, thus by \cite[Corollary 7.3.2.7]{Lur17} it is lax symmetric monoidal. 

\begin{prop} \label{prop:proper-tate-is-lax}
    The proper Tate construction $(-)^{\tau G} \colon \Sp^{BG} \to 
    \Sp$ has a canonical lax symmetric monoidal structure, such that the natural transformation $(-)^{hG} \Rightarrow (-)^{\tau G}$ becomes lax symmetric monoidal.
\end{prop}
\begin{proof}

    In \cite[Theorem I.3.1]{NS18}, it was shown that the canonical transformation $(-)^{hG} \Rightarrow (-)^{tG}$ (ordinary Tate construction) becomes lax symmetric with a unique choice of lax-symmetric structure on $(-)^{tG}$. We can mimic the proof by replacing the ordinary Tate construction with the proper Tate construction.

    In the proof in \cite{NS18}, the authors take $\scC = \Sp^{BG}$, and $\scD = \Sp^{BG}_{\text{Ind}}$, where $\Sp^{BG}_{\text{Ind}}$ is the stable subcategory of $\Sp^{BG}$ generated by spectra of the form $\bigoplus_{g \in G} X$, with $G$ acting by permutation -- these are exactly spectra induced from the subgroup $\{e\} \subseteq G$. The proof is completed formally, using some properties of the inclusion $\scD \subseteq \scC$, as well as the fact that the ordinary Tate construction $(-)^{t G}$ factors through $\scC/\scD$ (as the ordinary Tate construction of spectra induced from $\{e\}$ is $0$). 
    \\ \\
    We can instead let $\scD$ be a different stable subcategory of $\Sp^{BG}$, namely the smallest stable subcategory of 
    $\Sp^{BG}$ which includes all spectra with a $G$ action which is induced from any subgroup $H \subseteq G$. I.e., spectra of the form $\bigoplus_{gH \in G/H} gX$, where $X \in \Sp^{BH}$. Note that this is strictly larger than $\Sp^{BG}_{\text{Ind}}$, which is generated by spectra induced only from the trivial subgroup $\{e\}$.

    As in \cite{NS18}, to complete the proof, we need to show a parallel of \cite[Lemma I.3.8]{NS18} for $\scD$, which we will do after this proof in \cref{lem:closure-of-induced}.

    Then $(-)^{\tau G}$ factors through $\scC \to \scC/\scD$ as it kills all induced spectra (\cite[Remark 2.26]{Yuan21}), and the rest of the proof works in the same way. 
    
\end{proof}

\begin{lem} \label{lem:closure-of-induced}
    Let $X \in \Sp^{BG}$.
    \begin{enumerate}
        \item[(i)] If $X \in \scD$ then $X^{\tau G} = 0$. 
        \item[(ii)] For all $Y \in \scD$, $X \otimes Y \in \scD$ (i.e., $\scD$ is an $\otimes$-ideal). 
        \item[(iii)] The natural maps:
        \[
        \colim_{Y \in \scD/X} Y \to X
        \]
        and \[
        \colim_{Y \in \scD/X} \cofib(Y \to X)^{hG} \longrightarrow \colim_{Y \in \scD/X} \cofib(Y \to X)^{\tau G} = X 
        \]
        are equivalences. 
    \end{enumerate}
\end{lem}

\begin{proof}
For (i), by exactness of $(-)^{\tau G}$, the full subcategory where $(-)^{\tau G}$ vanishes is stable. Thus, it is enough to show that the proper Tate construction vanishes on induced spectra. This is discussed in \cite[Remark 2.26]{Yuan21}. 
\\ \\
For part (ii), given $X \in \Sp^{BG}$, the full subcategory of $Y$ such that $X \otimes Y \in \scD$ is again stable, so it suffices to assume $Y = \bigoplus_{gH \in G/H} gZ$ is induced. But in this case there is an $H$-equivariant inclusion map $X \otimes Z \to X \otimes Y$ given by inclusion of one of the summands. This map is adjoint to the $G$-equivariant map 
\[
\bigoplus_{gH \in G/H} g(X \otimes Z) \to X \otimes Y
\]
which is an equivalence. But the left-hand side is induced, so $X \otimes Y \in \scD$.
\\ \\
Finally, for (iii), first note that this colimit is filtered, as $\scD$ is stable.

For the first map, it is enough to check that for all $i \in \ZZ$, the map:
\[
\colim_{Y \in \scD/X} \pi_i Y \to \pi_iX
\]
is an isomorphism. This proves the first equivalence, as the colimit is filtered and thus commutes with $\pi_i$. By translation, we can assume $i = 0$. Let $\alpha \colon \SSS \to X$ represent an element in $\pi_0 X$. Then there is a $G$-equivariant map $Y :=\bigoplus_{g \in G} \SSS \to X$ defined such that on one of the copies of $\SSS$ it is given by $\alpha$ (by $G$-equivariance, this determines the whole map). Then the inclusion of this copy $\SSS \to Y$ is an element of $\pi_0(Y)$ that maps to $\alpha$.

For injectivity, suppose $Y \in \scD$, and $\beta \in \ker(\pi_0 Y \to \pi_0 X)$. Then there is a corresponding map $\SSS \to Y$ such that $\SSS \to Y \to X$ is $0$. Let $\overline{Y} := \cofib\left(\bigoplus_{g\in G} \SSS \to Y\right)$, which is an object in $\scD$ with an induced map $\overline{Y} \to X$ such that there is a factorization $Y \to \overline Y \to X$. By construction, $\beta$ dies in $\pi_0 \overline Y$, and as the colimit is filtered, this means it dies in the colimit. 
\\ \\
For the second map, note that for all $Y \in \scD/X$, there is a fiber sequence:
\[
\colim_{G/H} \cofib(Y \to X)^{hH} \to \cofib(Y \to X)^{hG} \to \cofib(Y \to X)^{\tau G} = X^{\tau G}
\]
(Where the last equality is true because $(-)^{\tau G}$ is exact, and $Y^{\tau G} = 0$). Passing to a filtered colimit over $Y$, the left term vanishes, as:
\[
\begin{split}
    \colim_{Y \in \scD/X} \left( \colim_{G/H} \cofib(Y \to X)^{hH} \right) & = \colim_{G/H} \left( \colim_{Y \in \scD/X} \cofib(Y \to X)^{hH} \right) \\
    & = \colim_{G/H} \left( \colim_{Y \in \scD/X} \cofib(Y \to X) \right)^{hH} \\ 
    & = \colim_{G/H} \left( \cofib\left( \colim_{Y \in \scD/X} Y \to X \right)\right)^{hH} \\ 
    & = 0
\end{split}
\]
Where in the second equality we used the fact that (small) limits, namely $(-)^{hH}$, commute with filtered colimits \cite[Proposition 5.3.3.3]{Lur09}, and the final equality is by the first equivalence of (iii). 
\end{proof}

\subsection{Tate Diagonal} \label{subsection:tate-diagonal}

Denote by $T_G \colon \Sp \to \Sp$ the functor $T_G(X) = \left(X^{\otimes G}\right)^{\tau G}$, where $G$ acts by permutation of the factors.

\begin{prop}
    For a finite group $G$, $T_G \colon \Sp \to \Sp$ is exact and lax symmetric monoidal.
\end{prop}

\begin{proof}
    Lax symmetric monoidality is immediate, as $(-)^{\otimes G}$ is lax symmetric monoidal, and $(-)^{\tau G}$ is too, as discussed in \cref{prop:tate-is-exact}. 
    \\ \\
    The proof of exactness will follow the same lines as the proof in \cite[Proposition III.1.1]{NS18} for the case $G = C_p$.

    First, let us see that $T_G$ preserves sums. Fix some $X_0, X_1 \in \Sp$. Note that $G$ acts on the set $\{0,1\}^G$ by left-translation, and let $P$ be the set of orbits of this action. In particular, we have the one element orbits $0=(0,...,0)$ and $1=(1,...,1)$. All other orbits are non-trivial (this is because the only fixed points of the action are $0$ and $1$, as the action of $G$ on itself is transitive). 
    
    We have:
    \[
    \begin{split}
    T_G(X_0 \oplus X_1) & = \left( \bigoplus_{(i_g)\in\{0,1\}^G} \left( \bigotimes_{g \in G} X_{i_g} \right) \right)^{\tau G} \\ 
    & = T_G(X_0) \oplus T_G(X_1) \oplus \bigoplus_{O \in P \setminus \{0,1\}} \left( \bigoplus_{(i_g) \in O} \bigotimes_{g \in G}X_{i_g} \right)^{\tau G}
    \end{split}
    \]

    For $O = O((i_{g})) \in P \setminus \{0,1\}$, the summand:
    \[
    \bigoplus_{(i_g) \in O} \bigotimes_{g \in G}X_{i_g} 
    \]
    is induced from the stabilizer $\text{stab}((i_g))$ of $(i_g)$, on the $\text{stab}((i_g))$-spectrum $\bigotimes_{g \in G} X_{i_g}$. Since the proper Tate construction vanishes on induced $G$-spectra (\cite[Remark 2.26]{Yuan21}), we conclude that $T_G(X_0 \oplus X_1) \simeq T_G(X_0) \oplus T_G(X_1)$. 
    \\\\
    To conclude exactness, it suffices to check that $T_G$ commutes with extensions. Indeed, suppose that that $X_0 \to \widetilde{X} \to X_1$ is a fiber sequence. Then we obtain a filtration of $\left( \widetilde{X} \right)^{\otimes G}$, whose successive filtration steps are given by $(X_0)^{\otimes G}$, and then:
    \[
    \bigoplus_{(i_g) \in O} \bigotimes_{g \in G} X_{i_g}
    \]
    ordered by the sum $\sum_{g \in G} i_g$. But because $(-)^{\tau G}$ is exact kills all non-trivial orbits, we get a fiber sequence:
    \[
    ((X_0)^{\otimes G})^{\tau G} \to ((\widetilde{X})^{\otimes G})^{\tau G} \to ((X_1)^{\otimes G})^{\tau G}
    \]
\end{proof}

\begin{cor} \label{cor:tate-diagonal}
    Let $G$ be a finite group. Then there is a unique lax symmetric natural transformation:
    \[
    \Delta_G \colon \Id_{\Sp} \Rightarrow T_G
    \]
    Which we call the (generalized) Tate-diagonal. 
\end{cor}

\begin{proof}
    $\Id_{\Sp}$ is initial among lax symmetric monoidal exact functors $\Sp \to \Sp$ (\cite[Corollary 6.9]{Nik16}). 
\end{proof}

Later, in \cref{prop:power-operations-factor-tate}, we will see that the Tate-diagonal $\Delta_G : X \to (X^{\otimes G})^{\tau G}$ does what we expect intuitively, namely it is a homotopy coherent version of the diagonal map on spaces $\Omega^{\infty} X \to \left( (\Omega^{\infty} X)^{\times G}\right)^{hG}$, which is not additive and therefore not a map of spectra.  

For a ring spectrum $R$ we have the Tate-valued Frobenius map:

\begin{definition}
    Let $R$ be a ring spectrum, and $G$ a finite group. Then the (generalized) Tate-valued Frobenius map, with respect to $G$, is the map:
    \[
    \phi^G \colon R \xrightarrow[]{\Delta_G} (R^{\otimes G})^{\tau G} \xrightarrow[]{\,m^{\tau G}\,} R^{\tau G}
    \]
    Where $m \colon R^{\otimes G} \to R$ is the multiplication map of $R$.
\end{definition}

As $\Delta_G$ is a lax symmetric transformation, multiplication is also lax symmetric, and $(-)^{\tau G}$ is a lax-symmetric functor, we obtain

\begin{cor} \label{cor:frob-is-lax}
    Let $G$ be a finite group. Then the Tate-valued Frobenius map assembles into a lax-symmetric natural transformation $\phi^G \colon \CAlg(\Sp) \Rightarrow \CAlg(\Sp)$.
\end{cor}

\begin{rmrk}
    There is also a relative version of the Tate-valued Frobenius, such that for $H \subseteq G$ it gives a map:
    \[
    \phi_H^G \colon R^{\tau H} \to R^{\tau G}
    \]
    We will not define this map here. \cref{thm:proper-tate-functoriality} gives a nice functorial property of the relative Frobenius. 
\end{rmrk}

\subsection{Generalized Frobenius Maps on Flat Ring Spectra} \label{subsection:gen-frob}

Let $R$ be a $p$-adically flat ring spectrum over the $p$-completed sphere spectrum. Recall that the trivial map $R \to R^{\tau A}$ is an equivalence for all abelian $p$-groups $A$. An immediate consequence is that, for all surjections $A \twoheadrightarrow B$, the canonical map $R^{\tau B} \to R^{\tau A}$ is an equivalence, as by \cref{thm:proper-tate-functoriality} we have a commutative diagram: 
\[\begin{tikzcd}
	R & {R^{\tau B}} & {R^{\tau A}}
	\arrow["{\can_1^B}"', from=1-1, to=1-2]
	\arrow["{\can_1^A}", bend left=18, from=1-1, to=1-3]
	\arrow["{\can_B^A}"', from=1-2, to=1-3]
\end{tikzcd}\]
And since $\can_1^A$ and $\can_1^B$ are equivalences, so is $\can_B^A$.

\begin{prop} \label{prop:composition-of-frobeni}
    Let $R$ be a $p$-adically flat ring, and $A$ a finite abelian $p$-group of order $|A| = p^m$. Then the generalized Frobenius map
    \[
     R \xrightarrow[]{\phi^{A}} R^{\tau A} \xrightarrow[]{(\can^A)^{-1}} R
    \]
    is equivalent to the $m$-fold composition of the  Frobenius map $R \xrightarrow[]{\phi}R^{tC_p} \xrightarrow[]{(\can^{C_p})^{-1}} R$.
\end{prop}

\begin{proof}
    We proceed by induction on $m$. For $m = 0$, the statement is immediate.

    For the induction step, we may choose some subgroup $A' \subseteq A$ of order $p^{m-1}$ (For instance, write $A \cong C_{p^k} \times B$ and choose $A' = C_{p^{k-1}} \times B \subseteq A$). Then we have the following exact sequence of abelian groups:
    \[
    0 \to A' \hookrightarrow A \twoheadrightarrow C_p \to 0
    \]
    We claim the following diagram commutes:
    \[\begin{tikzcd}[row sep = large, column sep = large]
	R & {R^{\tau A'}} & {R^{\tau A}} \\
	& R & {R^{tC_p}} \\
	&& R
	\arrow["{\phi^{A'}}", from=1-1, to=1-2]
	\arrow["{\phi^A}", bend left=30, from=1-1, to=1-3]
	\arrow["{\phi_{A'}^A}", from=1-2, to=1-3]
	\arrow["{\can^{A'}}", from=2-2, to=1-2]
	\arrow["{\phi^{C_p}}"', from=2-2, to=2-3]
	\arrow["{\can_{C_p}^A}", from=2-3, to=1-3]
	\arrow["{\can^A}"', bend right=40, from=3-3, to=1-3]
	\arrow["{\can^{C_p}}", from=3-3, to=2-3]
\end{tikzcd}\]
    The commutativity of the first row and the right column follows immediately from the functoriality of the (relative) Frobenii maps and the canonical maps, given in \cref{thm:proper-tate-functoriality}. The commutativity of the square follows from the fact that both compositions coincide in the span category:

    \[\begin{tikzcd}[column sep = small]
	&& {A'} &&&&&& {A'} \\
	& 1 & {} & A && \equiv && {A'} & {} & {A'} \\
	1 && {C_{p}} && A && 1 && {A'} && A
	\arrow[two heads, from=1-3, to=2-2]
	\arrow["\lrcorner"{anchor=center, pos=0.125, rotate=-45}, draw=none, from=1-3, to=2-3]
	\arrow[hook, from=1-3, to=2-4]
	\arrow[two heads, from=1-9, to=2-8]
	\arrow["\lrcorner"{anchor=center, pos=0.125, rotate=-45}, draw=none, from=1-9, to=2-9]
	\arrow[hook, from=1-9, to=2-10]
	\arrow[two heads, from=2-2, to=3-1]
	\arrow[hook, from=2-2, to=3-3]
	\arrow[two heads, from=2-4, to=3-3]
	\arrow[hook, from=2-4, to=3-5]
	\arrow[two heads, from=2-8, to=3-7]
	\arrow[hook, from=2-8, to=3-9]
	\arrow[two heads, from=2-10, to=3-9]
	\arrow[hook, from=2-10, to=3-11]
\end{tikzcd}\]
    Thus the composition $R \xrightarrow[]{\phi^{A}} R^{\tau A} \xrightarrow[]{(\can^A)^{-1}} R$ is the same as the composition:
    \[
    R \xrightarrow[]{\phi^{A'}} R^{\tau A'} \xrightarrow[]{(\can^{A'})^{-1}} R \xrightarrow[]{\phi^{C_p}} R^{tC_p} \xrightarrow[]{(\can^{C_p})^{-1}} R 
    \]
    By induction, $R \xrightarrow[]{\phi^{A'}} R^{\tau A'} \xrightarrow[]{(\can^{A'})^{-1}} R$ is the $(m-1)$-fold composition of the Frobenius endomorphism, which completes the proof.
    
\end{proof}

\section{Comparison with the Expected Functor} 
\label{section:correctness}

In this section we finally show that the construction we gave in \cref{section:construction} makes the following diagram commute:

\begin{equation} \label{diag:main-goal}
    \begin{tikzcd}[row sep=large, column sep = large] 
    \CAlg_{/\pcs}^{\text{flat, ff}} \arrow[r,  "(-) \cotimes E"] \arrow[d, "\pi_0"] &[4em] \CAlg_{E}^{K(n),\text{even}} \arrow[d, "\pi_0"] \\ 
    \Alg_{\delta, \text{torsion-free}} \arrow[r, "A \mapsto B^A"] & \Alg_{\TT}
\end{tikzcd}
\end{equation}

In \cref{subsection:recollection-on-T-algs} we recall the sheaf structure associated to the $\TT$-algebra $\pi_0R$, where $R$ is a $K(n)$-local $E$-algebra. 

In \cref{subsection:reduction-universal}, we specialize to the case of $p$-adically flat ring spectra, and reduce our problem to the universal case, i.e., we show that it it suffices to show that our construction and $\pi_0(R \cotimes E)$ agree on the universal deformations. These universal deformations correspond to power operations with respect to symmetric groups. Thus, we only need to understand these power operations.

In \cref{subsection:abelian-subgroups}, we reduce our problem once again to understanding the power operations on abelian subgroups of symmetric groups. In \cref{subsection:comp-proper-tate} we connect the abelian power operations with the Tate-valued Frobenius map by showing that the Tate-valued Frobenius map factors through the power operations. Specifically, we establish the following decomposition of the $\pi_0$ of the Tate-valued Frobenius map:
\[
\pi_0 \phi^A: \pi_0R \xrightarrow[]{\,\overline {P_A} \,} (R^0(BA)/(I_{tr}))^{\free} \xrightarrow[]{\, \overline{G_R} \,} \pi_0(R^{\tau A})
\]
where the power operation $P_A$ is modded by some transfer ideal and reduced to its free part; and $\overline{G_R}$ is a canonical map we show exists. 

In \cref{subsection:calculations}, we show that the map $\overline{G_R}$ is injective, and thus it is enough to understand the $\pi_0$ of the Tate-valued Frobenius in order to understand the power operations. Finally, in \cref{subsection:putting-it-all-together}, we determine the $\pi_0$ of the Tate-valued Frobenius and show that it agrees with what we expect.

\subsection{Recollections on $\TT$-algebras} \label{subsection:recollection-on-T-algs}
Suppose that $R$ is a $K(n)$-local $E$-algebra. Then $\pi_0R$ has the structure of a $\TT$-algebra. In \cite{Rezk09}, Rezk constructs a functor $\Alg_{\TT} \to \Sh(\Def, \Alg)$, which is fully faithful when restricted to $p$-torsion free algebras. In this chapter, we describe the image of $\pi_0R$ under this functor. 

We first introduce some notation:

\begin{definition}
    Let $(-)^{\free} \colon \text{Ring} \to 
\text{Ring}$ denote the functor on rings which quotients a ring by its $p$-torsion (this functor may also be defined on abelian groups). It is the left adjoint to the inclusion of $p$-torsion free rings in all rings. The quotient map $S \to S^{\free}$ is the unit of the adjunction.
\end{definition}

\begin{definition} \label{def:power-operation}
    For a subgroup $G \subseteq \Sigma_{p^m}$, we define $L[G] := (E^0(BG)/(I_{tr}))^{\free}$, where $I_{tr}$ is the transfer ideal in $E^0(BG)$. Explicitly, $I_{tr}$ is generated by the images of the transfer maps $E^0(BH) \to E^0(BG)$ for all proper subgroups $H \subsetneqq G$.

    This ring has two natural maps from $E_0$. The first is the trivial map $s[G] \colon E_0 \to L[G]$, which canonically makes $L[G]$ an $E_0$-algebra. 
    The second map comes from the power operation with respect to $G$,
    \[
    t[G] \colon E_0 \to \pi_0 ((E^{\otimes p^m})^{hG}) \to \pi_0(E^{hG}) = E^0(BG) \to E^0(BG) / (I_{tr}) \to (E^0(BG) / (I_{tr}))^{\free} = L[G]
    \]
    Where the first map is the diagonal, and the second map is multiplication in $E$. 
    
    % Both maps $s[G]$ and $t[G]$ are maps of topological rings, as they are induced from maps of $E$-algebras. That is, they respect the adic topology on $E_0$ and $L[G]$.  
    
    Note that the map $t[G]$ depends on the embedding $G \hookrightarrow \Sigma_{p^m}$. 

    In the case $G = \Sigma_{p^m}$, we will denote $L[G]=L[m]$, $s[G]=s[m]$, and $t[G] = t[m]$. If it is clear from context, we will suppress the $m$-th index from $s$ and $t$.
\end{definition}

\begin{rmrk} \label{rmrk:morava-e-of-b-sigma-is-free}
    For $G = \Sigma_{p^m}$, it is a theorem that $E^0(B\Sigma_{p^m}) / (I_{tr})$ is free over $E_0$ (\cite[Theorem 8.6]{Strickland98}). Hence,
    \[
    L[m] = \left(E^0(B\Sigma_{p^m}) / (I_{tr})\right)^{\free} = E^0(B\Sigma_{p^m}) / (I_{tr}) \]
\end{rmrk}

\noindent If $R$ is a $K(n)$-local $E$-algebra, then the power operations on $R$ give rise to maps:
\[
P_{m}: R^0(X) \to R^0(X \times B\Sigma_{p^m})
\]
Since $E^0 (B\Sigma_{p^m})$ is a finite free $E_0$-module (\cite[Theorem 3.2]{Strickland98}), by the Künneth formula in the $K(n)$-local category
\[
R^0(X \times B\Sigma_{p^m}) \cong R^0 X \cotimes[E^0] E^{0}(B\Sigma_{p^m}) 
\]
This isomorphism is natural, and in particular it commutes with the transfer map. Thus:
\begin{equation} \label{eq:kunneth-plus-transfer}
    R^0(X \times B \Sigma_{p^m})/(I_{tr}) \cong R^0(X) \cotimes[{E_0}] E^0(B \Sigma_{p^m})/(I_{tr})
\end{equation}
We can also define the reduced power operations:
\[
\overline{P_m} \colon R^0(X) \xrightarrow[]{\; P_m \;} R^0 X \cotimes[E^0] E^{0}(B\Sigma_{p^m})  \to R^0X \cotimes[E_0] L[m]
\]

\begin{thm}[{\cite[Proposition 3.24]{Rezk09}}]
    $\overline{P_m}$ is a ring homomorphism.
\end{thm}

Letting $X = *$, we obtain maps $\psi_m \colon \pi_0R \to \pi_0R \cotimes[E_0] L[m]$. Note that the maps $t[m]$ are precisely the maps $\overline{P_m}$ in the case $X = *$ and $R = E$, hence they are also ring homomorphisms. Note that the maps $s[m]$ and $t[m]$ also respects the adic topology, as they are induced from maps of $E$-modules.  
\\\\
We are now ready to describe how $\pi_0R$ becomes a sheaf on $\Def$, where $R$ is a $K(n)$-local $E$-algebra. The collection $\{L[m]\}_{m \ge 0}$ determines a comonad, such that being a comodule over it is equivalent to being a ring $A$ with structure maps $\psi_m \colon A \to A \cotimes[E_0] L[m]$ that are compatible with respect to comultiplication maps $L[m] \cotimes L[k] \to L[m + k]$. Then there is an equivalence between the category of these comodules and the category of $\TT$-algebras which satisfy the congruence criterion, in the $p$-torsion free case. We will describe how one can recover the sheaf structure associated to such a comodule.
\\ \\
First we need to understand the ring $L[m]$. This ring classifies deformations of $\Frob^m$. That is, maps $L[m] \to R$ uniquely correspond to deformations of $\Frob^m$ over the ring $R$. In particular, the map $s[m] \colon E_0 \to L[m]$ classifies the source of the deformation, and the map $t[m] \colon E_0 \to L[m]$ classifies the target. That is, letting $\GG = \Spf(E^*(\CP^{\infty})) = \Spf(E_0[\![x]\!])$ be the universal deformation, a map $g:L[m] \to R$ corresponds to a deformation of the form $(g \circ s)^*\GG \to (g \circ t)^*\GG$. 

% Let us describe the correspondence explicitly. 
%  Let $s = s[m],t=t[m] \colon E_0 \to L[m]$ be defined as before. 
% The universal deformation of $\Frob^m$ will be a morphism $f_m \colon s^*\GG \to t^*\GG$, with  corresponding ring map:
% \[
% (f_m)^* \colon \scO_{t^*\GG} = L[m]_t \otimes E_0[\![x]\!] \to L[m]_s \otimes E_0[\![x]\!]  = \scO_{s^*\GG}
% \]
% Which is just the identity map. 
% For any other deformation of $\Frob^m$ over a ring $R$, there is a map $L[m] \to R$, such that the base change of $f_m$ to $R$ gives the original deformation over $R$. 

In particular, suppose $R$ is a $K(n)$-local $E$-algebra, and let us describe the sheaf $A$ induced from $\pi_0 R$. For any deformation of the formal group $E_0 \to S$, it gives $A_{S}(G)= \pi_0R \cotimes[E_0] S$. Let $g \colon L[m] \to S$ represent a deformation of $\Frob^m$, of the form $(g \circ s)^*\GG  \to (g \circ t)^*\GG$. Base changing the power operation $\psi_m \colon \pi_0R \to \pi_0R \cotimes[E_0] L[m]$, we get a map:
\[
\widetilde{\psi_m} \colon \pi_0R \cotimes[E_0]  L[m]_t \to \pi_0R \cotimes[E_0]  L[m]_s
\]
Where the subscripts $s$ and $t$ specify the $E_0$-algebra structure of $L[m]$ (the standard one is given by $s$). To get $A_S(g)$, we need to base change the map $\widetilde{\psi_m}$ to $S$. Explicitly, $A_S(g)$ is given by:
\[
\begin{split}
    A_S((g \circ t)^*\GG) & = \pi_0R \cotimes[E_0] S_{g\circ t} \\
     & =  \pi_0R \cotimes[E_0]  L[m]_t \cotimes[L[m]] S \xrightarrow[]{ \widetilde{\psi_m} \otimes \Id_S} \pi_0R \cotimes[E_0]  L[m]_s \cotimes[L[m]] S \\ 
     & =  \pi_0R \cotimes[E_0] S_{g \circ s}  = A_S((g \circ s)^*\GG)
\end{split}
\]

\subsection{Reduction to Universal Deformations} \label{subsection:reduction-universal}

Let us now specialize to the case $R = R' \cotimes E$, where $R'$ is $p$-adically flat.  Let $A$ be the sheaf associated to $\pi_0(R) = \pi_0(R') \cotimes E_0$. 
We recall that the goal of \cref{section:correctness} is to show that $A = B^{\pi_0(R')}$.

The power operation $\widetilde{\psi_m}$ takes the form:
\[
\widetilde{\psi_m} \colon \pi_0(R) \cotimes[E_0]  L[m]_t = \pi_0(R') \cotimes L[m] \to \pi_0(R) \cotimes[E_0] L[m]_s = \pi_0(R') \cotimes L[m]
\]
For a formal group $G$ over $S$, as expected, 
\[
A_S(G) = \pi_0(R' \cotimes E) \cotimes[E_0] S = \pi_0(R') \cotimes E_0 \cotimes[E_0] S =  \pi_0 (R') \cotimes S = B^{\pi_0(R')}_S(G)
\]
And for a deformation of $\Frob^m$ represented by a map $g \colon L[m] \to S$, we get:
\[
    A_S(g) \colon \pi_0R \cotimes[E_0] S = \pi_0R \cotimes[E_0] L[m]_t \cotimes[L[m]] S \xrightarrow[]{\widetilde{\psi_m} \otimes \Id_S} \pi_0 R \cotimes[E_0] L[m]_s \cotimes S = \pi_0(R) \cotimes[E_0] S
\]
The construction $B^{\pi_0 (R')}$ gives the map:
\[
B^{\pi_0 (R')}_S(g) \colon \pi_0(R') \cotimes S \xrightarrow[]{\phi^m \cotimes \Id_S }  \pi_0(R') \cotimes S
\]
where $\phi \colon \pi_0 (R') \to \pi_0(R')$ is the of the lift of Frobenius on $\pi_0(R')$. 

Thus, our goal is to show that 
$\widetilde{\psi_m} = \phi^m \otimes \Id_{L[m]}$. For this, it is enough to show that the original $\psi_m \colon \pi_0(R') \cotimes E_0 \to \pi_0(R') \cotimes L[m]$ is equal to $\phi^m \cotimes t[m]$, 
since after base change along the structure map $t[m] \colon E_0 \to L[m]$ we get $\widetilde{\psi_m} = \phi^m \cotimes \Id_{L[m]}$. 

Thus we obtain the following result:

\begin{lem} \label{lem:enough-to-show-two-are-equal}
    If we show that the maps $\psi_m \colon \pi_0 R \cotimes E_0 \to \pi_0R  \cotimes L[m]$ and the map $\phi^m \cotimes t[m]$ agree, then we can deduce that diagram (\ref{diag:main-goal}) commutes, as desired. 

    In other words, we need to see that the (reduced) $p^m$-power operation of $R \cotimes E$ is the same as the power operation of $E$ tensored with the $m$-th iteration of the Tate-valued Frobenius of $R$.
\end{lem}

\begin{proof}
   By the previous discussion, this data determines the functor $\Alg_{\delta} \to \Sh(\Def, \Alg)$ on objects, and shows that the diagram commutes on objects. Commutativity on morphisms is immediate, since morphisms of $\TT$-algebras are just morphisms of the underlying rings (\cref{rmrk:maps-of-gamma-rings}), where both functors correspond to (completed) base-change to $E_0$. 
\end{proof}
Suppressing the index $m$ from $t[m]$, we need to show $\psi = \phi^m \cotimes t \colon \pi_0(R') \cotimes E_0 \to \pi_0(R') \cotimes L[m]$.

\subsection{Abelian Subgroups} \label{subsection:abelian-subgroups}

Our next step is to show that the power operation is determined by its restriction to abelian subgroups of $\Sigma_{p^m}$. That is, we replace the power operations $\pi_0R \to \pi_0R \cotimes L[m]$ with restricted power operations $\pi_0R \to \pi_0R \cotimes L[A]$, where $A$ is an abelian $p$-group. These were defined in \cref{def:power-operation}. Our main tool will be \cref{cor:injective-to-abelian}.

For any subgroup $A \subseteq \Sigma_{p^m}$, there is a restriction map $R^0(B\Sigma_{p^m}) \to R^0(BA)$. We will need the following result:

\begin{thm}[\cite{Rezk09}]
    The transfer ideal in $R^0(B\Sigma_{p^m})$ is generated by the images of the transfer maps from subgroups $\Sigma_{p^m-p^k} \times \Sigma_{p^k}$ for $0 \le k \le m-1$.
\end{thm}
In particular, the transfer ideal is generated by transfers from subgroups acting non-transitively. 

\begin{lem}
    Let $A \subsetneqq \Sigma_{p^m}$ be a proper abelian subgroup acting transitively. 
    The image of the transfer ideal in $R^0(B\Sigma_{p^m})$ under the restriction $R^0(B\Sigma_{p^m}) \to R^0(BA)$ lands inside the transfer ideal of $R^0(BA)$.
\end{lem}

\begin{proof}
This follows from the double coset formula: the restriction to $A$ of a transfer from a subgroup $K$ is a direct sum of transfer from subgroups of the form $A \cap gKg^{-1}$, where $g \in \Sigma_{p^m}$. Since we know we may assume $K$ acts non-transitively, and hence so does $gKg^{-1}$, it follows that $A \cap gKg^{-1} \subsetneqq A$, so the transfer is proper and lands inside the transfer ideal of $R^0(BA)$. 
\end{proof}

Thus we obtain a map:

\begin{equation} \label{map:injective-to-abelian}
    R^0(B\Sigma_{p^m})/(I_{tr}) \to \bigoplus_{A\subseteq \Sigma_{p^m}} R^0(BA)/(I_{tr}) \to \bigoplus_{A\subseteq \Sigma_{p^m}} \left( R^0(BA)/(I_{tr}) \right)^{\free}
\end{equation}

where the sum ranges over abelian subgroups of order $p^m$ acting transitively. 

\begin{thm} \label{thm:injective-to-abelian}
    When $R = E$, the map (\ref{map:injective-to-abelian}) is an injection. 
\end{thm}

This is stated in \cite{HS20}, but for completion we will also prove this later, in \cref{subsection:calculations}. 

\begin{cor} \label{cor:injective-to-abelian}
    Let $R = R' \cotimes E$, where $R'$ is $p$-adically flat over $\SSS$. Then the map:
    \[
    R^0(B\Sigma_{p^m})/(I_{tr}) \to \bigoplus_{A\subseteq \Sigma_{p^m}} \left( R^0(BA)/(I_{tr}) \right) \to \bigoplus_{A\subseteq \Sigma_{p^m}} \left( R^0(BA)/(I_{tr}) \right)^{\free}
    \]
is injective, where $A$ ranges again ranges over abelian subgroups of order $p^m$, acting transitively. 
\end{cor}
\begin{proof}

    First, note that $E^0(B\Sigma_{p^m})/(I_{tr})$ is free over $E_0$, in particular $p$-torsion free (see \cite{Rezk09}). Hence 
\[R^0(B\Sigma_{p^m})/(I_{tr} ) = \pi_0(R') \cotimes  E^0(B\Sigma_{p^m})/(I_{tr})
    \]
    is also $p$-torsion free.
    
    Finally, we use injectivity of the map (\ref{map:injective-to-abelian}) which holds for $E$, along with base change to $\pi_0R' \cotimes E_0$, which is flat over $E_0$, thus preserves injectivity, and we use the fact that $\pi_0(R')$ is $p$-torsion free, hence tensoring with it commutes with the $(-)^{\free}$ functor.
\end{proof}
Next, we consider the following diagram:
\[\begin{tikzcd}
	{\pi_0R} & {R^0(B\Sigma_{p^m})} & {R^0(B\Sigma_{p^m}) / (I_{tr})} & {(R^0(B\Sigma_{p^m}) / (I_{tr}))^{\free}} \\
	& {R^0(BA)} & {R^0(BA)/(I_{tr})} & {\left( R^0(BA)/(I_{tr}) \right)^{\free}}
	\arrow["{{P_m}}", from=1-1, to=1-2]
	\arrow["{{P_A}}"', from=1-1, to=2-2]
	\arrow[from=1-2, to=1-3]
	\arrow[from=1-2, to=2-2]
	\arrow[equals, from=1-3, to=1-4]
	\arrow[from=1-3, to=2-3]
	\arrow[from=1-4, to=2-4]
	\arrow[from=2-2, to=2-3]
	\arrow[from=2-3, to=2-4]
\end{tikzcd}\]
To see this diagram commutes, we observe the following commutative diagram:
\[\begin{tikzcd}
	{\Omega^\infty R} & {\Omega^\infty \left(R^{\otimes p^m}\right)^{h\Sigma_{p^m}} } & {\Omega^\infty R^{h\Sigma_{p^m}}} \\
	{\Omega^\infty R} & {\Omega^\infty \left(R^{\otimes p^m}\right)^{hA} } & {\Omega^{\infty}R^{hA}}
	\arrow["\Delta", from=1-1, to=1-2]
	\arrow["m", from=1-2, to=1-3]
	\arrow[from=1-2, to=2-2]
	\arrow[from=1-3, to=2-3]
	\arrow[equals, from=2-1, to=1-1]
	\arrow["\Delta", from=2-1, to=2-2]
	\arrow["m", from=2-2, to=2-3]
\end{tikzcd}\]

This diagram commutes by the naturality of $(-)^{h\Sigma_{p^m}} \Rightarrow (-)^{hA}$. Post-composition with the first row gives $P_m$, and with the second is $P_A$ (since $A$ acts transitively, $|A| = p^m$). 

Specializing to the case $R = R' \cotimes E$, we see that the power operation $\psi_m \colon \pi_0(R' \cotimes E) \to \pi_0(R' \cotimes E) \cotimes[E_0] L[m]$ restricts to the power operation $P_A$ of $R$. Similarly, $t[m]$ restricted to $A$ is $t[A]$.

Combining the results of \cref{subsection:reduction-universal}, we get:

\begin{cor} \label{cor:enough-to-show-two-are-equal-A}
    Let $R = R' \cotimes E$, where $R'$ is a $p$-adically flat ring spectrum.
    
    Suppose we know that for all $A \subseteq \Sigma_{p^m}$ abelian and transitive, we have $\overline{P_A} = \phi^m \cotimes t[A]$, where $P_A$ is the power operation of $R$, and $t[A]$ is the power operation of $E$. 

    Then $\psi_m = \phi^m \cotimes t$, and thus, by \cref{lem:enough-to-show-two-are-equal}, the diagram (\ref{diag:main-goal}) commutes.
\end{cor}

\begin{proof}
    Since the map of \cref{cor:injective-to-abelian} is injective, it is enough to verify the equality $\psi_m = \phi^m \cotimes t[m]$ after composing with it.

    In particular, it is enough to check that these maps agree after composing with each of the restriction maps $R^0(B\Sigma_{p^m})/(I_{tr}) \to \left( R^0(BA)/(I_{tr}) \right)^{\free}$. 
    
    By the discussion above, the power operation $\psi_m$ becomes $\overline{P_A}$. 
    
    Similarly, $\phi^m \otimes t[m]$ becomes $\phi^m \otimes t[A]$: this is because the map $R^0(B\Sigma_{p^m})/(I_{tr}) \to \left( R^0(BA)/(I_{tr}) \right)^{\free}$ is the base change of the map $E^0(B\Sigma_{p^m})/(I_{tr}) \to \left( E^0(BA)/(I_{tr}) \right)^{\free}$ to $\pi_0(R') \cotimes E_0$, by (\ref{eq:kunneth-plus-transfer}), and the fact that $\pi_0(R') \cotimes E_0$ is $p$-torsion free. 
\end{proof}
That is, we need to show the following diagram commutes:

\[\begin{tikzcd}
	{\pi_0(R' \cotimes E)} & ({(R' \cotimes E)^0(BA)/(I_{tr})})^{\free} \\
	{\pi_0(R') \cotimes E_0} & {\pi_0 (R') \cotimes L[A]}
	\arrow["{\overline{P_A}}", from=1-1, to=1-2]
	\arrow[equals, from=1-1, to=2-1]
	\arrow[equals, from=1-2, to=2-2]
	\arrow["{\phi^m \cotimes t[A]}"', from=2-1, to=2-2]
\end{tikzcd}\]

\subsection{Composition with Proper Tate} \label{subsection:comp-proper-tate}

Let $R$ be a ring spectrum. Let $A \subseteq \Sigma_{p^m}$ be some abelian subgroup of order $p^m$. Suppose $A$ is transitive. Note that $|A| = p^m$, by transitivity and commutativity of $A$.

We recall that we are interested in the power operation $\overline{P_A} \colon \pi_0 R \to \left( R^0(BA)/(I_{tr})\right)^{\free}$. Our goal in the next two subsections is to factor the $\pi_0$ of the Tate-valued Frobenius map $\pi_0R \to \pi_0(R^{\tau A})$ through $\overline{P_A}$, and to show that the resulting map $\left( R^0(BA)/(I_{tr})\right)^{\free} \to \pi_0(R^{\tau A})$ is injective in the case that $R = R' \cotimes E$, where $R'$ is $p$-adically flat.

We begin with the factorization of the non-reduced power operation.

\begin{prop} \label{prop:power-operations-factor-tate}
    Let $A$ be any finite group (not necessarily abelian). The lax symmetric monoidal transformation:
    \[
    \Omega^\infty(\Delta_A) \colon \Omega^\infty \Rightarrow \Omega^\infty T_A
    \]
    factors as:
    \[
    \Omega^\infty R \xrightarrow[]{\Delta_{\Omega^\infty R}} \left( (\Omega^\infty R)^{\times A}\right)^{hA} \to \Omega^\infty\left( (R^{\otimes A})^{hA} \right) \to \Omega^\infty \left( (R^{\otimes A})^{\tau A} \right)
    \]
    Where $T_A(R)=(R^{\otimes A})^{\tau A}$. 
\end{prop}

\begin{proof}
    Both are lax symmetric monoidal transformations $\Omega^\infty \to \Omega^\infty T_A$. But $\Omega^\infty$ is initial among lax symmetric monoidal functors $\Sp \to \scS$ (\cite[Corollary 6.9]{Nik16}), thus they agree. 
\end{proof}

\begin{cor}
    Let $A$ be a finite group, and $R$ a ring spectrum. Let $\phi \colon R \to R^{\tau A}$ be the Tate-valued Frobenius map. Then $\pi_0 \phi \colon \pi_0 R \to \pi_0(R^{\tau A})$ factors as:
    \[
    \pi_0R \xrightarrow[]{P_A} \pi_0(R ^{hA}) \to \pi_0(R^{\tau A}) 
    \]
    where $P_A : \pi_0 R \to \pi_0(R^{hA})$ is the power operation with respect to $A$. 
\end{cor}
\begin{proof}
    We have the following commutative diagram:
    \[\begin{tikzcd}[column sep=large]
	{\Omega^{\infty} R} & {\Omega^{\infty}((R^{\otimes A})^{\tau A})} & {\Omega^{\infty} (R^{\tau A})} \\
	{((\Omega^{\infty}R)^{\times A})^{hA}} & {\Omega^{\infty}((R^{\otimes A})^{hA})} & {\Omega^{\infty} (R^{hA})}
	\arrow["{{\Omega^{\infty}\Delta_A}}", from=1-1, to=1-2]
	\arrow["{{\Omega^{\infty}\Delta}}"', from=1-1, to=2-1]
	\arrow["m", from=1-2, to=1-3]
	\arrow[from=2-1, to=2-2]
	\arrow[from=2-2, to=1-2]
	\arrow["m"', from=2-2, to=2-3]
	\arrow[from=2-3, to=1-3]
\end{tikzcd}\]
    Where the left rectangle commutes by the previous proposition, \cref{prop:power-operations-factor-tate}, and the right rectangle commutes by naturality of $(-)^{hA} \Rightarrow (-)^{\tau A}$. 

    Taking $\pi_0$, the composition of the top row becomes $\pi_0 \phi$, and the bottom row $p_A$. 
\end{proof}

Recall that the proper Tate construction fits in the cofiber sequence:
\[
\colim_{G/H \in \text{Orb}_G}  R^{hH} \to R^{hG} \to R^{\tau G}
\]
where $\text{Orb}_G$ is the orbit category of $G$. In particular, since the transfer map $R^{hH} \to R^{hG}$ factors through the colimit, the composition:
\[
R^0(BH) \xrightarrow[]{\tr} R^0(BG) = \pi_0(R^{hG}) \xrightarrow[]{} \pi_0(R^{\tau G})
\]
is the zero map. Thus the canonical map $\pi_0(R^{hA}) \to \pi_0(R^{\tau A})$ factors canonically through the quotient by the transfer ideal.

\begin{cor}
    For a spectrum $R$ with a $G$ action, the $\pi_0$ of the canonical map $R^{hG} \to R^{\tau G}$ factors as:
    \[
    \pi_0(R^{hG}) \to \pi_0(R^{hG})/(I_{tr}) \xrightarrow[]{G_R} \pi_0(R^{\tau G})
    \]
    Where $G_R : \pi_0(R^{hG})/(I_{tr}) \to \pi_0(R^{\tau A})$ is natural in $R$. When $R$ is a ring spectrum, $G_R$ is a multiplicative map. 
\end{cor}
\begin{proof}
    The existence is discussed above. Naturality (respectively, in the case that $R$ is a ring spectrum, multiplicativity)  is a consequence of the fact that the canonical map $R^{hG} \to R^{\tau G}$ and the quotient map $\pi_0(R^{hG}) \to \pi_0(R^{hG})/(I_{tr})$ are natural (respectively, multiplicative).
\end{proof}

\begin{cor} 
    Let $A$ be a finite group, and $R$ a ring spectrum. We have the following decomposition of the Tate-valued Frobenius:
    \[\begin{tikzcd}
	{\pi_0R} & {\pi_0(R^{hA})=R^0(BA)} & {R^0(BA)/(I_{tr})} \\
	& {\pi_0(R^{\tau A})}
	\arrow["{{P_A}}", from=1-1, to=1-2]
	\arrow["{{\pi_0 \phi_A}}"', from=1-1, to=2-2]
	\arrow[from=1-2, to=1-3]
	\arrow[from=1-2, to=2-2]
	\arrow["{G_R}", from=1-3, to=2-2]
\end{tikzcd}\]
\end{cor}
Next we show this decomposition is compatible with the lax structure of the proper Tate:

\begin{prop} \label{prop:lax-of-F-no-free}
    The map:
    \[
    G_R\colon \pi_0(R^{BA})/(I_{tr})\to \pi_0(R^{\tau A})
    \]
    is compatible with the lax symmetric monoidal structure of the functor $(-)^{\tau A}$ in the following way: for spectra $R,S$, the following diagram commutes
    \[\begin{tikzcd}[column sep = large]
	{\pi_0((R \otimes S)^{BA})/(I_{tr})} & {\pi_0((R \otimes S)^{\tau A})} \\
	{\pi_0(R) \otimes \pi_0(S^{BA})/(I_{tr})} & {\pi_0(R^{\tau A}) \otimes \pi_0(S^{\tau A})}
	\arrow["{G_{R \otimes S}}", from=1-1, to=1-2]
	\arrow[from=2-1, to=1-1]
	\arrow["{\can_R \otimes G_S}"', from=2-1, to=2-2]
	\arrow[from=2-2, to=1-2]
\end{tikzcd}\]
    Where the left map is induced by the canonical map $R \otimes S^{BA} \to (R \otimes S)^{BA}$, which itself is induced by the equivalence $\Sigma^{\infty}_+BA \xrightarrow[]{\sim} \SSS \otimes \Sigma^{\infty}_+BA $.
\end{prop}
\begin{proof}
    The canonical map $R^{BA} \to R^{\tau A}$ is lax symmetric monoidal  by \cref{prop:proper-tate-is-lax}, hence so is the map $\pi_0(R^{BA}) \to \pi_0(R^{\tau A})$. 

    Suppose $H \subseteq A$ is a proper subgroup. The transfer map $S^0(BH) \to S^0(BA)$ may be realized as being induced from the Becker-Gottlieb transfer $t:\Sigma^\infty_+BA \to \Sigma^\infty_+BH$. Thus we have the following commutative square:
    \[\begin{tikzcd}
	{(R \otimes S)^{BH}} & {(R \otimes S)^{BA}} \\
	{R \otimes S^{BH}} & {R \otimes S^{BA}}
	\arrow["{t^*}", from=1-1, to=1-2]
	\arrow[from=2-1, to=1-1]
	\arrow["{\Id_R \otimes t^*}", from=2-1, to=2-2]
	\arrow[from=2-2, to=1-2]
\end{tikzcd}\]

    After applying $\pi_0$ (along with its lax symmetric structure), we see that the transfer from $BH$ in $\pi_0(R) \otimes \pi_0(S^{BA})$ lands inside the transfer ideal in $\pi_0\left((R \otimes S)^{BA}\right)$. This means that the following diagram is well defined (and commutative):

    \[\begin{tikzcd}
	{\pi_0((R \otimes S)^{BA})} & {\pi_0\left( (R \otimes S)^{BA} \right)/(I_{tr})} \\
	{\pi_0(R) \otimes \pi_0(S^{BA})} & {\pi_0(R) \otimes \left( \pi_0(S^{BA})/(I_{tr})\right)}
	\arrow[from=1-1, to=1-2]
	\arrow[from=2-1, to=1-1]
	\arrow[from=2-1, to=2-2]
	\arrow[from=2-2, to=1-2]
\end{tikzcd}\]
    Denoting $L_A(R) =  \pi_0(R^{BA})/(I_{tr})$, we observe the followimg diagram
\[\begin{tikzcd}
	& {L_A(R \otimes S)} \\
	{\pi_0((R \otimes S)^{BA})} && {\pi_0((R \otimes S)^{\tau A})} \\
	& {\pi_0(R) \otimes L_A(S)} \\
	{\pi_0(R)\otimes \pi_0(S^{BA})} && {\pi_0(R^{\tau A}) \otimes \pi_0(S^{\tau A})}
	\arrow[from=1-2, to=2-3]
	\arrow[from=2-1, to=1-2]
	\arrow[from=2-1, to=2-3]
	\arrow[shorten <=7ex, from=3-2, to=1-2]
	\arrow[shorten >=7ex, no head, from=3-2, to=1-2]
	\arrow[from=3-2, to=4-3]
	\arrow[from=4-1, to=2-1]
	\arrow[from=4-1, to=3-2]
	\arrow[from=4-1, to=4-3]
	\arrow[from=4-3, to=2-3]
\end{tikzcd}\]
    The front face commutes as $(-)^{BA} \Rightarrow (-)^{\tau A}$ is lax symmetric monoidal, and the left face commutes by the previous diagram. Moreover, since the map $\pi_0(S^{BA}) \to L_A(S)$ is surjective (and the tensor product $\pi_0(R) \otimes (-)$ preserves surjectivity), the commutativity of the two squares implies the commutativity of the third.

\end{proof}
As we are interested in power operations to $(R^0(BA)/(I_{tr}))^{\free}$, we define:

\begin{definition} \label{def:the-map-F}
    For a spectrum $R$ with an $A$-action, we define a map 
    \[
    \overline{G_R} = \overline{G_R^A} \colon \left(R^0(BA)/(I_{tr})\right)^{\free} \to \pi_0(R^{\tau A})^{\free}
    \]
    to be the application of the $(-)^{\free}$ functor on the map $G_R$. Note that $\overline{G_R}$ is natural in $R$ since $G_R$ is natural in $R$. 
\end{definition}
Let us now specialize to the case $R = R' \cotimes E$, where $R'$ is $p$-adically flat, and let $A$ be abelian. Then $R^0(BA) = \pi_0(R') \cotimes E^0(BA)$ is $p$-torsion free. In \cite[Corollary 5.10]{HKR00}, it is shown that $E^0(BA)$ is torsion free, and since $\pi_0(R')$ is $p$-adically flat over $\ZZ_p$, completed tensor product with it does not introduce $p$-torsion. Hence $\pi_0(R^{\tau A})$ is also $p$-torsion free, since by \cref{thm:tate-of-complex-oriented} it is a localization of $R^0(BA)$.

Hence the map $G_R \colon R^0(BA)/(I_{tr}) \to \pi_0(R^{\tau A})$ actually factors as:
\[
G_R \colon R^0(BA)/(I_{tr}) \to \left( R^0(BA)/(I_{tr})\right)^{\free} \xrightarrow[]{\;\overline{G_R}\;} \pi_0(R^{\tau A})^{\free} = \pi_0(R^{\tau A})
\]

\begin{cor}
    \label{cor:decomposition-of-frob}
    Let $R = R' \cotimes E$, where $R'$ is a $p$-adically flat ring spectrum. Then the following diagram commutes:
    \[\begin{tikzcd}
	{\pi_0R} & {\pi_0(R^{hA})=R^0(BA)} & {R^0(BA)/(I_{tr})} & {\left(R^0(BA)/(I_{tr})\right)^{\free}} \\
	& {\pi_0(R^{\tau A})}
	\arrow["{{{P_A}}}", from=1-1, to=1-2]
	\arrow["{{{\pi_0 \phi_A}}}"', from=1-1, to=2-2]
	\arrow[from=1-2, to=1-3]
	\arrow[from=1-2, to=2-2]
	\arrow[from=1-3, to=1-4]
	\arrow["{{G_R}}"'{pos=0.7}, from=1-3, to=2-2]
	\arrow["{{\overline{G_R}}}", from=1-4, to=2-2]
\end{tikzcd}\]
\end{cor}

\begin{cor} \label{cor:lax-of-F}
    Since the $(-)^{\free}$ construction is symmetric monoidal, \cref{prop:lax-of-F-no-free} implies the following diagram commutes, for ring spectra $R,S$,

    \[\begin{tikzcd}[column sep=large]
	{\left( \pi_0((R \otimes S)^{BA})/(I_{tr}) \right)^{\free}} & {\pi_0((R \otimes S)^{\tau A})^{\free}} \\
	{\pi_0(R)^{\free} \otimes \left( \pi_0(S^{BA})/(I_{tr}) \right)^{\free}} & {\pi_0(R^{\tau A})^{\free} \otimes \pi_0(S^{\tau R})^{\free}}
	\arrow["\overline{{G_{R \otimes S}}}", from=1-1, to=1-2]
	\arrow[from=2-1, to=1-1]
	\arrow["{{\can_R \otimes \overline{G_S}}}"', from=2-1, to=2-2]
	\arrow[from=2-2, to=1-2]
\end{tikzcd}\]
    
\end{cor}

\subsection{Morava E-theory of Abelian $p$-Groups} \label{subsection:calculations}

Let $R = R' \cotimes E$, where $R'$ is a $p$-adically flat ring spectrum. Fix an abelian group $A$, and let $X = \Spf(E_0)$. 

The goal of this section is to prove that the canonical map
\[
\overline{G_R} \colon \left(R^0(BA)/(I_{tr})\right)^{\free} \to \pi_0(R^{\tau A})^{\free} = \pi_0(R^{\tau A})
\]
constructed in the previous subsection, is injective. Our approach is algebro-geometric.  Geometric interpertations of $\left(R^0(BA)/(I_{tr})\right)^{\free}$ are well established in the literature (see, for example, \cite{AHS04} and \cite{HS20}), and we recall them here. We then develop a geometric interpretation of $\pi_0(R^{\tau A})$, which we use to analyze the map $\overline{G_R}$.  
% While related constructions appear in the literature, we are not aware of any explicit geometric interpretation of $\pi_0(R^{\tau A})$ in this form.
\\ \\
Let $A^* = \Hom(A, \CC^{\times})$ be the dual group of $A$. An element $a \in A$ can be viewed as a character of $A^*$ (via evaluation), and thus it gives rise to a line bundle $V_a$ on $BA^*$, i.e., a map $BA^* \to BS^1$. 
This map corresponds to an $R^0(BA^*)$-point of the formal group $G = \Spf(R^0(BS^1))$. That is, we have a map $\chi \colon A \to \Gamma(\Spf( R^0(BA^*)), G)$. 

Since $V_{a + b} = V_a \otimes V_b$, the map $\chi$ is a group homomorphism. Thus it is classified by a scheme homomorphism:
\[
\widetilde{\chi} \colon \Spf(R^0(BA^*)) \to \Hom(A, G)
\]
\begin{thm}[{\cite[Proposition 7.3]{AHS04}}] \label{thm:E-cohomology-of-abelian-groups}
    Let $R = R' \cotimes E$, where $R'$ is a $p$-adically flat ring $R'$; and let $A$ be a finite abelian group. The map $\widetilde{\chi} \colon \Spf(R^0(BA^*)) \to \Hom(A,G)$ is an isomorphism. 
\end{thm}

In particular, if $A = C_{p^{m_1}} \times \cdots \times C_{p^{m_k}}$, 
we have:
\begin{equation} \label{eq:group-cohomology}
    R^0(BA^*) = \frac{\pi_0R[\![x_1, ..., x_k]\!]}{\left( [p^{m_1}](x_1),...,[p^{m_k}](x_k) \right)}
\end{equation}
Where $x_i$ is the first Chern class of the line bundle $V_{a_i}$, for $a_i$ some choice of a generator of $C_{p^{m_i}}$. Next we find $\pi_0(R^{\tau A^*})$.

\begin{prop} \label{prop:pi-0-of-proper-Tate}
    Let $A = C_{p^{m_1}} \times \cdots \times C_{p^{m_k}}$. Let $S \subseteq \pi_0(R^{BA^*})$ be the multiplicative set generated by expressions of the form:
    \[
    [i_1](x_1) +_F [i_2](x_2) +_F \cdots +_F [i_k](x_k)
    \]
    where $1 \le i_1 \le p^{m_1}-1$, $1 \le i_2 \le p^{m_2}-1$, $\ldots$, $1 \le i_k \le p^{m_k}-1$, and $+_F$ is the sum with respect to the formal group law. Then $\pi_0(R^{\tau A^*}) = S^{-1} \pi_0(R^{BA^*})$.
\end{prop}

\begin{proof}
We will use \cref{thm:tate-of-complex-oriented} to calculate $\pi_0(R^{\tau A^*})$ 

Let $L_i = V_{a_i}$ be the representation of $A^*$ corresponding to $a_i$. Then the tensor products $\{L_1^{\otimes i_1} \otimes \cdots \otimes L_{k}^{\otimes i_k}\}$ are the irreducible representations of $A^*$. For the reduced regular representation, we sum all of these except:
\[
1 = L_1^{\otimes 0} \otimes \cdots \otimes L_{k}^{\otimes 0}
\]

Then we have: 
\[
\begin{split}
    e(\widetilde{\rho_{A^*}}) & = e\left( \bigoplus_{i_1,...,i_k \ne 0,...,0} L_1^{\otimes i_1} \otimes \cdots \otimes L_{k}^{\otimes i_k}\right) \\
    & = \prod_{i_1,...,i_k \ne 0,...,0}  e\left( L_1^{\otimes i_1} \otimes \cdots \otimes L_{k}^{\otimes i_k} \right) \\
    & = \prod_{i_1,...,i_k \ne 0,...,0}  \left( e\left( L_1^{\otimes i_1} \right) +_F \cdots +_F \,e\left( L_{k}^{\otimes i_k} \right) \right) \\ 
    & = \prod_{i_1,...,i_k \ne 0,...,0}  \left( [i_1]\left( e(L_1) \right) +_F \cdots +_F [i_k]\left(e( L_{k}) \right) \right) \\ 
    & = \prod_{i_1,...,i_k \ne 0,...,0}  \left( [i_1]\left( x_1 \right) +_F \cdots +_F [i_k]\left(x_k \right) \right) \\ 
\end{split}
\]
Inverting this product is equivalent to inverting each factor individually.
\end{proof}

Let us now specialize to the case of Morava E-theory. Let $\GG := \Spf(E^0(\CP^{\infty}))$ be the universal deformation, which is a formal group scheme over $X = \Spf (E_0)$.
Let $A = C_{p^{m_1}} \times \cdots \times C_{p^{m_k}}$. Fix an integer $M$ such that $M > m_1,...,m_k$. We can define a (formal) subscheme $\GG[A] \subseteq \GG$ classifying homomorphisms $A \to \GG$. Since all elements of $A$ have order dividing $p^M$, every such homomorphism factors through $\GG[C_{p^M}]$. We know the ring of functions of $\GG[C_{p^M}]$ is:
\[
\scO_{\GG[C_{p^M}]} = \frac{E_0[\![x]\!]}{([p^M](x))}
\]
As the height of the group is $n$, in $[p^M](x)$, the first invertible coefficient is that of $x^{p^{Mn}}$, with all prior coefficients lying in the maximal ideal.  
Thus $E_0[\![x]\!]/([p^M](x))$ is actually finite and free. This is a consequence of the Weierstrass Preparation lemma (\cite[chapter VII, §3, Proposition 6]{Bourbaki72}).
Explcitly, we can write $[p^M](x) = u_M(x) P_M(x)$, where $u_M(x)$ is an invertible power series; $P_M(x)$ is a monic polynomial of degree $p^{Mn}$, with all coefficients except for the coefficient of $x^{p^{Mn}}$ lying in the maximal ideal of $E_0$; and $u_M,P_M$ are unique.

A monic polynomial such that all of its coefficients are in the maximal ideal, except for the leading coefficient, such as $P_M$, is called \emph{distinguished}.

We thus have:
\[
\frac{E_0[\![x]\!]}{([p^M](x))} = \frac{E_0[\![x]\!]}{(P_M(x))}= \frac{E_0[x]}{(P_M(x))}
\]
And this ring is freely generated as an $E_0$-module by $1,x,...,x^{p^{Mn}-1}$. This follows from \cite[chapter VII, §3, Proposition 5]{Bourbaki72}. 
\\\\
In particular, we may extend its functor of points to also allow non-nilpotent elements,
\begin{definition}
    Let $X = \Spf(E_0)$, and let $\GG$ be the formal group over $X$ associated to Morava $E$-theory. We define the $p^M$-torsion in $\GG$ to be the $X$-scheme:
    \[
\GG[p^M] = \Spec_{X}\left( \frac{E_0[\![x]\!]}{([p^M](x))} \right) = \Spec_X\left(  \frac{E_0[x]}{(P_M(x))} \right)
\]
    By the previous discussion, $\GG[p^M]$ is finite over $X$. In particular, the formal group law becomes a polynomial in $E_0[x]/(P_M(x))$, thus $\GG[p^M]$ inherits a structure of a finite group scheme over $X$.
\end{definition}

\begin{rmrk}
    Note that since we extended the functor of points, there is no morphism $\GG[p^M] \to \GG$. That is, the quotient map $E_0[\![x]\!] \to E_0[\![x]\!]/(P_M(x))$ does not give a map $\GG[p^M] \to \GG$, as points of $\GG[p^M]$ allow non-nilpotent images for $x$.
    We may define $\GG[p^M]_{\rest}$ as the closed subscheme of $\GG$ cut by $[p^M](x)$. Then there is a closed embedding $\GG[p^M]_{\rest} \hookrightarrow \GG$, and a morphism $\GG[p^M]_{\rest} \to \GG[p^M]$ ``forgetting" about the nilpotency. 
\end{rmrk}

We now define a subscheme $\GG[A] \subseteq \GG[p^M]^k$ classifying homomorphisms $A \to \Gamma(Y, \GG[p^M])$ for an $X$-scheme $Y$. This is the extension of the way we previously defined $\GG[A]$, allowing non-nilpotents. 

Explicitly, for $A = C_{p^{m_1}} \times \cdots \times C_{p^{m_k}}$,
\[
\GG[A] = \Spec_X \left(  \frac{E_0[\![x_1,\ldots,x_k]\!]}{\left([p^{m_1}](x_1),\ldots,[p^{m_k}](x_k) \right) } \right)
\]
Where we use, again, that the ring above is finite and flat because $[p^{m_i}](x_i)$ has an invertible coefficient for ${x_i}^{p^{m_i n}}$.
Further, we have a closed embedding $\GG[A] \hookrightarrow \GG[p^M]^k$ (where $k$ is the rank of $A$), which restricts to the generators. This map is non-canonical and depends on a choice of generators.

We also denote $\AAA^1 := \AAA^1_X= \Spec_X(E_0[x])$. There is closed embedding $\GG[p^M]\hookrightarrow \AAA^1$, given on rings by the quotient map $E_0[x] \to E_0[x]/(P_M(x))$. 
Note that $\AAA^1$ does not inherit a structure of group scheme, since the formal group law of $\GG$ need not be given by a polynomial.

\begin{claim}
    The closed embedding $\GG[p^M] \to \GG[p]$, has corresponding map on rings $E_0[x]/(P_1(x)) \to E_0[x]/(P_M(x))$, sending $x \mapsto x$.
\end{claim}
\begin{proof}
    This is obvious from the following commutative diagram:
    \[\begin{tikzcd}
	{E_0[\![x]\!] / ([p](x))} & {E_0[\![x]\!] / ([p^M](x))} \\
	{E_0[x]/(P_1(x))} & {E_0[x]/(P_M(x))}
	\arrow["{x \mapsto x}", from=1-1, to=1-2]
	\arrow["{x\mapsto x}"', from=1-1, to=2-1]
	\arrow["\simeq", from=1-1, to=2-1]
	\arrow["{x\mapsto x}", from=1-2, to=2-2]
	\arrow["{\simeq }"', from=1-2, to=2-2]
	\arrow[from=2-1, to=2-2]
\end{tikzcd}\]
\end{proof}
We will denote $P(x) := P_1(x)$, and $u(x) := u_1(x)$. That is, $[p](x) = u(x)P(x)$, where $u(x)$ is an invertible power series, and $P(x)$ is a distinguished polynomial of degree $p^n$. 

We will use the notation $O(x^r)$ to mean a polynomial, or power series, which divides $x^r$, i.e., of the form $a_rx^r + a_{r+1}x^{r+1} + \cdots$. 

\begin{claim}
    The polynomial $P_M(x)$ satisfies $P_M(x) = u\cdot p^M x +O(x^2)$, where $u \in (E_0)^{\times}$
\end{claim}
\begin{proof}
    We know that $[p^M](x) = p^M x + O(x^2)$. $u(x)$ is invertible, thus $u_M(x) = u_0 + u_1x + O(x^2)$ with $u_0 \in (E_0)^{\times}$. Suppose $P_M(x) = a_0 + a_1x + O(x^2)$. Then:
    \[
    [p^M](x) = p^M x + O(x^2) = u_M(x)P_M(x) = u_0a_0 + (u_1a_0+u_0a_1)x + O(x^2) \in E_0[\![x]\!]
    \]
    Hence $u_0a_0 = 0$, which implies $a_0 = 0$. Therefore $p^M = u_0 a_1$, i.e. $a_1 = (u_0)^{-1}p^M$.  
\end{proof}

Let us introduce a useful piece of notation:

\begin{definition} \label{def:x-of-point}
    Let $Y$ be an $X$-scheme, and $p \in \Gamma(Y, \GG[p^M])$ be a $Y$-point of $\GG[p^M]$. Then we denote image of $p$ under the closed embedding $\GG[p^M] \to \AAA^1$ by $x(p) \in \AAA^1(Y)\simeq \scO_{Y}$. 
\end{definition}

We now introduce two subschemes of $\GG[A]$, each classifying maps $A \to \Gamma(Y, \GG[p^M])$ that behave in a way somewhat analogous to injectivity.

\begin{definition}
    Let $A = C_{p^{m_1}} \times \cdots \times C_{p^{m_k}}$ be an abelian $p$-group of order $p^{m_1}\cdots p^{m_k} = p^m$. Let $\GG[A] \subseteq \GG[p^M]^k$ be defined as before. Explicitly,
    \[
\GG[A] = \Spec_X \left(  \frac{E_0[\![x_1,\ldots,x_k]\!]}{\left([p^{m_1}](x_1),\ldots,[p^{m_k}](x_k) \right) } \right)
\]
    We define the following subschemes of $\GG[A]$:
    \begin{enumerate}
        \item[-] We let $\Mon(A,\GG)$ be the open subscheme defined by inverting expressions of the form 
        \[
        [i_1]\left( x_1 \right) +_F \cdots +_F [i_k]\left(x_k \right) 
        \]
         Where $1 \le i_1 \le p^{m_1}-1$, $1 \le i_2 \le p^{m_2}-1$, $\ldots$, $1 \le i_k \le p^{m_k}-1$. Recall that $\GG[p^M]$ is a finite group scheme, and by $[i_j]$ we mean the $i_j$-th series in $\GG[p^M]$. 
         
    \item[-] Let $\Level(A,\GG)$ be defined such that its functor of points maps a scheme $Y$ to $A$-level structures on $\GG$ over $Y$. A level $A$-structure over $Y$ is a map $\phi \colon A \to \Gamma(Y, \GG[p^M])$, such that:
    \begin{equation} \label{eq:divisor-inequality}
        [\phi A[p]] :=\sum_{a \in A[p]} [x(\phi(a))] \le \GG[p] \times_X Y
    \end{equation}
    Where $A[p]$ is the $p$-torsion in $A$; $[x(\phi(a))]$ is the divisor defined by the image of $\phi(a)$ in $\AAA^1$, compatible with the notaion of \cref{def:x-of-point}; and $\GG[p]$ is the $p$-torsion of $\GG$, considered as a divisor of $\AAA^1$. 
    That is, we require this divisor inequality to hold in $\AAA^1$.  

    Note that $\GG[p] \subseteq \AAA^1$ is the divisor defined exactly by the polynomial $P_1(x)=P(x)$, as $\GG[p] = \Spec_X(E_0[x]/(P(x)))$.
    \end{enumerate}
\end{definition}

\begin{rmrk}
    In \cite{Strickland97}, $\Level(A, \GG)$ is defined slightly differently. There, a map $\phi : A \to \Gamma(Y, \GG)$ is a level $A$-structure if it satisfies the same divisor inequality (\ref{eq:divisor-inequality}), but interpreted as an inequality of divisors on $\GG$. We will refer to this definition by the notaion $\Level(A,\GG)_{\rest}$.
    
    As $\AAA^1$ may be seen as the extension of the functor of points of $\GG$ to non-nilpotent elements (forgetting the group structure), our definition can be seen as a natural extension of Strickland’s to this non-nilpotent setting.
    
    Essentially, the condition of being a level structure has two components: first, the requirement to be a group homomorphism, which is detected on $\GG[p^M]$; and second, the divisor inequality, which is verified in $\AAA^1$.
\end{rmrk}

We again stress that these schemes are subschemes of $\GG[A]$, but not subschemes of $\GG$. 
These subschemes enjoy a number of useful properties.

\begin{lem}[{\cite[Proposition 7.2.]{Strickland97}}] For an abelian $p$-group $A$, and a formal group $\GG$, $\Level(A,\GG) \subseteq \GG[A]$ is a closed subscheme. 

\end{lem}

\begin{lem} \label{lem:same-global-funs}
    Let $\Level(A, \GG)_{\rest}, \GG[A]_{\rest}$ denote the schemes $\Level(A, \GG), \GG[A]$, with the functor of points restricted to nilpotents. Then $\scO_{\GG[A]_{\rest}} = \scO_{\GG[A]}$ and $\scO_{\Level(A, \GG)_{\rest}} = \scO_{\Level(A,\GG)}$. 
\end{lem}
\begin{proof}
    The equality $\scO_{\GG[A]_{\rest}} = \scO_{\GG[A]}$ is by definition. 

    Let us explain why $\scO_{\Level(A, \GG)_{\rest}} = \scO_{\Level(A,\GG)}$ holds. 
    Let $\Div(\AAA^1)$ be the scheme classifying of divisors of $\AAA^1$. Then there is a closed subscheme $E \hookrightarrow \Div(\AAA^1)$ which classifies divisors $D \in \Div(\AAA^1)(Y)$ such that $D \le \GG[p] \times_X Y$ (divisor inequality is a closed condition). 

    There is a morphism $\GG[A] \to \Div(\AAA^1)$, such that a homomorphism $\phi \colon A \to \Gamma(Y, \GG[p^M])$ is mapped to $[\phi A[p]]$. Then, by definition:
    \[
    \Level(A, \GG) = \GG[A] \times_{\Div(\AAA^1)} E
    \]
    Identically for the restricted functors, we have:
     \[
    \Level(A, \GG)_{\rest} = \GG[A]_{\rest} \times_{\Div(\AAA^1)} E
    \]
    Thus:
    \[
    \scO_{\Level(A,\GG)} = \scO_{\GG[A]} \otimes_{\scO_{\Div(\AAA^1)}} \scO_{E} = \scO_{\GG[A]_{\rest}} \otimes_{\scO_{\Div(\AAA^1)}} \scO_{E} = \scO_{{\Level(A,\GG)}_{\rest}}
    \]
    
\end{proof}

\begin{lem}
    $\Mon(A,\GG)$ classifies maps $\phi \colon A \to \Gamma(Y, \GG[p^M])$ which are injective everywhere. That is, $x(\phi(a)) \in \scO_Y$ is invertible for all $a \ne 0$. See \cref{def:x-of-point} for the notation $x(\phi(a))$.
\end{lem}
\begin{proof}
    This follows from the definition of $\Mon(A,\GG)$; by construction, it is formed by inverting all expressions of the form $[i_1](x_1) +_F [i_2](x_2) +_F \cdots +_F [i_k](x_k)$.
\end{proof}

The discussion above may be summarized by the following diagram:

\[\begin{tikzcd}
	{\Mon(A,\GG)} \\
	{\Level(A,\GG)} & {\GG[A]} & {\GG[p^M]^k} & {\AAA^k}
	\arrow[from=1-1, to=2-2]
	\arrow[from=2-1, to=2-2]
	\arrow[from=2-2, to=2-3]
	\arrow[from=2-3, to=2-4]
\end{tikzcd}\]
We now relate the diagram to topology:
\begin{prop} \label{prop:topology-of-level} 
    Let $A$ be an abelian $p$-group, and $\GG$ the formal group associated to Morava $E$-theory.
    Then
    $\scO_{\Level(A, \GG)} = \left(E^0(BA^*)/(I_{tr})\right)^{\free}$, with the quotient map $E^0(BA^*) \to \left(E^0(BA^*)/(I_{tr})\right)^{\free}$ corresponding to the closed embedding  $\Level(A, \GG) \to \GG[A]$. 
\end{prop}

\begin{proof}
    We again adopt the notation $\Level(A,\GG)_{\rest}$ and $\GG[A]_{\rest}$ for the functor of points restricted to nilpotents. 

    In {{\cite[Proposition 5.2]{HS20}}}, it was shown that the quotient map $q:E^0(BA^*) \to \left(E^0(BA^*)/(I_{tr})\right)^{\free}$ corresponds to the closed embedding of formal schemes $\Level(A,\GG)_{\rest} \to \GG[A]_{\rest}$.

    Note the following commutative diagram:
    \[\begin{tikzcd}
	{\Level(A,\GG)_{\rest}} & {\GG[A]_{\rest}} \\
	{\Level(A,\GG)} & {\GG[A]}
	\arrow[hook, from=1-1, to=1-2]
	\arrow[from=1-1, to=2-1]
	\arrow[from=1-2, to=2-2]
	\arrow[hook, from=2-1, to=2-2]
\end{tikzcd}\]
    Applying the functor $\scO_{(-)}$, and using \cref{lem:same-global-funs}, we get:

    \[\begin{tikzcd}
	{\scO_{\Level(A,\GG)_{\rest}}} & {\scO_{\GG[A]_{\rest}}} \\
	{\scO_{\Level(A,\GG)}} & {\scO_{\GG[A]}}
	\arrow[equals, from=1-1, to=2-1]
	\arrow["q"', from=1-2, to=1-1]
	\arrow[equals, from=2-2, to=1-2]
	\arrow[from=2-2, to=2-1]
\end{tikzcd}\]

    Thus the inclusion $\Level(A,\GG) \to \GG[A]$ also corresponds to the quotient map $q:E^0(BA^*) \to (E^0(BA^*)/(I_{tr}))^{\free}$. In particular, $\scO_{\Level(A, \GG)} = \left(E^0(BA^*)/(I_{tr})\right)^{\free}$.
    
\end{proof}

We have a similar result for $\Mon(A,\GG)$.

\begin{prop} \label{prop:topology-of-mon}
    Let $A$ be an abelian $p$-group, and $\GG$ be the formal group associated to Morava $E$-theory.
    Then $\scO_{\Mon(A,\GG)} = \pi_0(E^{\tau A^*})$, and the open embedding $\Mon(A,\GG) \to \GG[A]$ corresponds to the map of rings $E^0(BA^*) \to \pi_0(E^{\tau A^*})$. 
\end{prop}
\begin{proof}
    By definition, $\scO_{\Mon(A,\GG)}$ is obtained from $\scO_{\GG[A]}$ by inverting all expressions of the form \[
    [i_1]\left( x_1 \right) +_F \cdots +_F [i_k]\left(x_k \right)
    \]
    In \cref{prop:pi-0-of-proper-Tate} we showed that this is the localization $E^0(BA^*) \to \pi_0(E^{\tau A^*})$. 
\end{proof}

\begin{prop}
    Let $A$ be an abelian $p$-group, and $\GG$ be the formal group associated to Morava $E$-theory. Then the open embedding $\Mon(A,\GG) \to  \GG[A]$ factors as $\Mon(A,\GG) \to  \Level(A,\GG) \to \GG[A]$. Further, the map $\Mon(A,\GG) \to  \Level(A,\GG)$ is induced by the homomorphism $\overline{G_E} \colon \left(E^0(BA^*)/(I_{tr})\right)^{\free} \to \pi_0(E^{\tau A^*})$ on rings.
\end{prop}

\begin{proof}
    By \cref{prop:topology-of-mon}, we know that map $E^0(BA^*) \to \pi_0(E^{\tau A^*})$ corresponds to the open embedding $\Mon(A,\GG) \to \GG[p^M]$. 
    By \cref{prop:topology-of-level}. map $E^0(BA^*) \to \left( E^0(BA^*)/(I_{tr}) \right)^{\free}$ corresponds to the closed embedding $\Level(A,\GG) \to \GG[A]$. 

    Recall that by definition of $\overline{G_E}$, we have the following commutative diagram:
    \[\begin{tikzcd}[column sep=tiny]
	{\pi_0(E^{BA^*})} && {\pi_0(E^{\tau A^*})} \\
	& {\left(E^0(BA^*)/(I_{tr})\right)^{\free}}
	\arrow[from=1-1, to=1-3]
	\arrow[from=1-1, to=2-2]
	\arrow["{{\overline{G_E}}}"', from=2-2, to=1-3]
\end{tikzcd}\]
    (This diagram appears right before \cref{cor:decomposition-of-frob}). Taking $\Spec$, we get:
    \[\begin{tikzcd}[column sep=tiny]
	{\GG[A]} && {\Mon(A,\GG)} \\
	& {\Level(A,\GG)}
	\arrow[hook', from=1-3, to=1-1]
	\arrow["{{(\overline{G_E})^{*}}}", from=1-3, to=2-2]
	\arrow[hook', from=2-2, to=1-1]
\end{tikzcd}\]

    Which imply that there is an open embedding $\Mon(A,\GG) \to  \Level(A,\GG)$ which corresponds to the map of rings $\overline{G_E}$. 
\end{proof}

We recall that $X = \Spf(E_0)$. 
If $Y$ is an $X$-scheme, we will denote by $Y_{\QQ}$ the \emph{rationalization} of $Y$ -- the base change of $Y$ to $X_{\QQ} := X \times \Spec(\QQ) = \Spf(E_0 \otimes \QQ)$. 

\begin{prop} \label{prop:iso-after-rationalization}
    Let $A$ be an abelian $p$-group, and $\GG$ be the formal group associated to Morava $E$-theory.
    The inclusion $\Mon(A,\GG) \to \Level(A,\GG)$ becomes an isomorphism after rationalization.
\end{prop}

\begin{proof}

    % First note that $(\Mon(A,\GG[p^M]))_{\QQ} = \Mon(A,(\GG[p^M])_{\QQ})$. This is because for any rational test scheme $Y$, $\Mon(A,\GG[p^M])(Y)$ classifies maps $A \to \Gamma(Y, \GG[p^M]) = \Gamma(Y, (\GG[p^m])_{\QQ})$ which are injective everywhere. 

    % Similarly, $(\Level(A,\GG[p^M]))_{\QQ} = \Level(A,(\GG[p^M])_{\QQ})$. Again, if $Y = \Spec_X(S)$ is rational, i.e., $S$ is a $E_0 \otimes \QQ$-algebra, then the condition on the divisors (\ref{eq:divisor-inequality}) happens in $\GG[p] \times_X Y = \GG[p]_{\QQ} \otimes $, which is rational if $Y$ is rational). 

    Let $Y = \Spec_X(S)$ be a rational test scheme, where $X = \Spf(E_0)$. That is, $S$ is an ($E_0 \otimes \QQ$)-algebra. We want to show that 
    \[
    \Mon(A,\GG)(Y) \to \Level(A,\GG)(Y)
    \]
    is bijective. It is trivially injective  (both sides classify maps $A \to \Gamma(Y,\GG[p^M])$), and we now show surjectivity.
    Suppose that $\phi \colon A \to \Gamma(Y,\GG[p^M])$ is a level structure. We claim that to check that $x(\phi(a)) \in \scO_Y = S$ is invertible for all $a \in A$, it is enough to check that $x(\phi(a))$ is invertible for $a \in A[p]$, where $x(\phi(a))$ is defined in \cref{def:x-of-point}.
    
    This is because for any $a \in A$, for some integer $k$, we have $0 \ne p^ka \in A[p]$, and 
    \[
    \begin{split}
        x(\phi(p^ka)) & = [p^k](x(\phi(a))) = u_k(x(\phi(a))P_k(x(\phi(a))) \\
        & = u_k(x(\phi(a)) \cdot \left( u \cdot p^k x (\phi(a)) + O(x(\phi(a)) ^ 2)\right) \\ 
        & = u_k(x(\phi(a)) \cdot u \cdot p^k \cdot x (\phi(a)) \cdot\left( 1 +  O(x(\phi(a)))\right)
    \end{split}
    \] 
    Where $u \in (E_0)^{\times}$.
    Here, $[p^k]$ is the $p^k$-series inside $\GG[p^M]$.
    Note that $u_k(x)$ was defined to be a power series, not a polynomial, but here we mean the image of $u_k(x)$ in $E_0[\![x]\!]/([p^M](x)) \simeq E_0[x]/(P_M(x))$. 

    If we show that $x(\phi(p^ka))$ is invertible in $S$, it will follow that $x(\phi(a))$ is invertible as well, as the former is a multiple of the latter.
\\\\
    Thus, let $a \in A[p]$. Since $[\phi A[p]] \le \GG[p]$, we have:
    \[
    x ( x - x(\phi(a)) ) \mid f_{[\phi A[p]]} \mid f_{\GG[p]}(x) = P(x) = u\cdot px + O(x^2) \in E_0[x]
    \]
    Where $u\in (E_0)^{\times}$. As $x$ is a non-zero divisor, this implies that $x - x(\phi(a)) \mid up + O(x)$. Thus the leading term $-x(\phi(a))$ must divide the leading term $up$.
    Since $u p$ is invertible, $x(\phi(a))$ is invertible as well.
\end{proof}

\begin{thm} \label{thm:injection-to-tate}
Let $R = R' \cotimes E$, where $R'$ is a $p$-adically flat ring spectrum. Let $A$ be an abelian $p$-group, of order $p^m$ and rank at most $n$ (i.e., it is a direct product of at most $n$ cyclic groups).

Then the canonical homomorphism:
    \[
    \overline{G_R} \colon \left(R^0(BA)/(I_{tr})\right)^{\free} \to \pi_0(R^{\tau A})
    \]
    is injective.
\end{thm}
\begin{proof}
    Since we know $R$ has the form $R' \cotimes E$, where $R'$ is $p$-adically flat, we have:
    \[
    \left(R^0(BA)/(I_{tr})\right)^{\free} = (\pi_0(R') \cotimes E_0) \otimes_{E_0} \left(E^0(BA)/(I_{tr})\right)^{\free} 
    \]
    And similarly, by direct calculation (using \cref{thm:tate-of-complex-oriented}):
    \[
    \pi_0(R^{\tau A}) = (\pi_0 (R') \cotimes E_0) \otimes_{E_0} \pi_0(E^{\tau A})
    \]
    (here it is important to complete first and tensor with $\pi_0(E^{\tau A})$ second!). Since $\overline{G_R}$ is a localization at the Euler class, direct computation shows the factor $(\pi_0(R') \cotimes E_0)$ is unchanged by the map $\overline{G_R}$, i.e. we have:
    \[
    \overline{G_R} = (\pi_0(R') \cotimes E_0) \otimes_{E_0} \overline{G_E}
    \]
    Since $\pi_0(R') \cotimes E_0$ is flat over $E_0$ (in the classical algebraic sense), as $\pi_0(R')$ is $p$-adically flat, it is enough to check injectivity of the map 
    \[
    \overline{G_E} \colon \left(E^0(BA)/(I_{tr})\right)^{\free} \to \pi_0(E^{\tau A})
    \]
    Since both rings are $p$-torsion free, it is enough to show that this is true after rationalization. But after rationlization, this map corresponds to the isomorphism $\Mon(A^*,G)_{\QQ} \xrightarrow[]{\simeq} \Level(A^*,G)_{\QQ}$ of \cref{prop:iso-after-rationalization}, in particular the map of rings is injective. 
\end{proof}

% We also get a corollary:
% \begin{cor}
%     Let $R = R' \cotimes E$, where $R'$ is $p$-adically flat. Then the canoncial map:
%     \[
%     R^0(B\Sigma_{p^m}
%     )/(I_{tr}) \to \pi_0(R^{\tau\Sigma_{p^m}})
%     \]
%     is injective.
% \end{cor}
% \begin{proof}
%     We have the following commutative diagram:
%     \[\begin{tikzcd}
% 	{R^0(B \Sigma_{p^m})/(I_{tr})} & {\bigoplus_{A \subseteq \Sigma^{p^m}}(R^0(BA)/(I_{tr})})^{\free} \\
% 	{\pi_0(R^{\tau \Sigma_{p^m}})} & {\bigoplus_{A \subseteq \Sigma^{p^m}}\pi_0(R^{\tau A})}
% 	\arrow[from=1-1, to=1-2]
% 	\arrow[from=1-1, to=2-1]
% 	\arrow[from=1-2, to=2-2]
% 	\arrow[from=2-1, to=2-2]
% \end{tikzcd}\]
%     Where the sums range over abelian transitive $A \subseteq \Sigma_{p^m}$. 
%     It is clear this diagram commutes without quotient by the transfer ideal, and thus it also commutes after the quotient. 

%     The top map in injective by \cref{cor:injective-to-abelian}, and the right map is injective by \cref{thm:injection-to-tate}. Hence the map on the left is also injective, as desired. 
% \end{proof}

Finally, we return to proving \cref{thm:injective-to-abelian}:
\begin{proof}[Proof of \cref{thm:injective-to-abelian}]
    By \cite[Theorem 9.2]{Strickland98}, we know that $\spec( E^0(B\Sigma_{p^m})/(I_{tr
    })) =\Sub_{m}(\GG) $, which is the scheme that classifies subgroups of $\GG$ of order $p^m$.

   In \cite{Strickland97}, Strickland introduces a closed subscheme for an abelian $p$-group $A$:
    \[
    \Type(A,\GG) := \Level(A,\GG) / \text{Aut}(A)
    \]
    By \cite[Theorem 7.4]{Strickland97}, given a level structure $\phi \colon A \to \Gamma(Y,\GG[p^M])$, then $[\phi A]$ is a subgroup of $\Gamma(Y,\GG[p^M])$. This induces a map $\Level(A,\GG) \to \Sub_m(\GG)$ (the subgroup $[\phi A]$ is the image of $\phi$). This morphism corresponds to the restriction:
    \[
    E^0(B\Sigma_{p^m})/(I_{tr}) \to \left( E^0(BA^*)/(I_{tr}) \right)^{\free}
    \]
    This was proven in \cite[Proposition 8.3]{HS20} in more general settings. 
    \\ \\ 
    Clearly we have $[\phi A] = [(\phi \circ f) A]$, for $f \in \Aut(A)$, thus this map descends to an evident map $\Type(A,\GG) \to \text{Sub}_{m}(\GG)$.

    Letting $A$ run over abelian $p$-subgroups of $\Sigma_{p^m}$ of order $p^m$ and rank at most $n$, we obtain a morphism
    \[
    \prod_{A} \Type(A,\GG) \to \text{Sub}_{m}(\GG)
    \]
    which, in \cite[Theorem 12.4]{Strickland97}, is shown to induce an injection on rings. The quotient map $\Level(A,\GG) \to \Type(A,\GG)$ is always induces injections on rings: it is the inclusion of the ring of fixed elements in the ambient ring. Hence the composition
    \[
    \prod_{A} \Level(A,\GG) \to \prod_{A} \Type(A,\GG) \to \text{Sub}_{m}(\GG)
    \]
    is injective on rings, which is the injection that we need. 
\end{proof}

\subsection{Completing the Proof} \label{subsection:putting-it-all-together}

Fix some abelian $p$-group $A$, and for a ring spectrum $R$, let us fix the notation $\phi_R$ for the Frobenius map $R \to R^{\tau A}$.

Let $R$ be $p$-adically flat and $A$ an abelian $p$-group. Our goal in this chapter is to complete the proof that diagram (\ref{diag:main-goal}) commutes, by finally showing the premise of \cref{cor:enough-to-show-two-are-equal-A} holds for $R\cotimes E$. That is, we show:
\begin{prop} \label{prop:power-agrees-with-frob}
    For $R$ a $p$-adically flat ring spectrum, the two maps
    \[
        \overline{P_A},\;\phi^m \cotimes t[A] \colon \pi_0(R \cotimes E) \to \left( (R \cotimes E)^0(BA)/(I_{tr}) \right)^{\free}
    \]
    agree.
\end{prop}

\begin{proof}
    
In \cref{thm:injection-to-tate} we saw that the map 
\[
\overline{G_{R \cotimes E}} \colon \left( (R \cotimes E)^0(BA)/(I_{tr})\right)^{\free} \to \pi_0((R \cotimes E)^{\tau A})
\]
is injective. Thus to prove \cref{prop:power-agrees-with-frob}, we can show the two maps agree after composing with the map $\overline{G_{R \cotimes E}}$. Further, the canonical map $\pi_0(R) \otimes \pi_0(E) \to \pi_0(R \cotimes E)$ is the completion at the maximal ideal of $E_0$, by \cref{cor:pi-0-of-flat-tensor-E}. In particular, since our maps are maps of complete local rings, they agree if and only if they agree on $\pi_0(R) \otimes \pi_0(E)$. 
That is, it suffices to show the following diagram commutes:
\begin{equation} \label{diag:after-pre-and-post}
\begin{tikzcd}[row sep = tiny]
	& {\pi_0(R \cotimes E)} & {\left( (R \cotimes E)^0(BA)/(I_{tr}) \right)^{\free}} \\
	{\pi_0(R) \otimes \pi_0(E)} &&& {\pi_0((R \cotimes E)^{\tau A})} \\
	& {\pi_0(R \cotimes E)} & {\left( (R \cotimes E)^0(BA)/(I_{tr}) \right)^{\free}}
	\arrow["{\overline{P_A}}", from=1-2, to=1-3]
	\arrow[from=1-3, to=2-4]
	\arrow[from=2-1, to=1-2]
	\arrow[from=2-1, to=3-2]
	\arrow["{\phi^m \cotimes t[A]}", from=3-2, to=3-3]
	\arrow[from=3-3, to=2-4]
\end{tikzcd}
\end{equation}
The composition:
\[
    \pi_0(R \cotimes E) \xrightarrow[]{\,\overline{P_A}\,} \left( (R \cotimes E)^0(BA)/(I_{tr}) \right)^{\free} \xrightarrow[]{\, \overline{G_{R \cotimes E}} \,} \pi_0((R \cotimes E) ^{\tau A})
\]
is the Frobenius map $\phi_{R \cotimes E}$, by \cref{cor:decomposition-of-frob}. Note the following diagram, which commutes since the Tate-valued Frobenius is natural and lax symmetric monoidal:
\[\begin{tikzcd}[column sep = large]
	{\pi_0(R \cotimes E)} & {\pi_0((R \cotimes E)^{\tau A})} \\
	{\pi_0(R \otimes E)} & {\pi_0((R \otimes E)^{\tau A})} \\
	{\pi_0(R) \otimes \pi_0(E)} & {\pi_0(R^{\tau A}) \otimes \pi_0(E^{\tau A})}
	\arrow["{{{\phi_{R \cotimes E}}}}", from=1-1, to=1-2]
	\arrow[from=2-1, to=1-1]
	\arrow["{\phi_{R \otimes E}}", from=2-1, to=2-2]
	\arrow[from=2-2, to=1-2]
	\arrow[from=3-1, to=2-1]
	\arrow["{{{\phi_{R} \otimes \phi_{E}}}}", from=3-1, to=3-2]
	\arrow[from=3-2, to=2-2]
\end{tikzcd}\]
Hence, in diagram~(\ref{diag:after-pre-and-post}), the composite along the upper path is the map
\[
\pi_0(R) \otimes \pi_0(E) \xrightarrow[]{\phi_R \otimes \phi_E} \pi_0(R^{\tau A}) \otimes \pi_0(E^{\tau A}) \to \pi_0((R \cotimes E)^{\tau A})
\]
\\ 
We now turn to the lower path. Consider the following commutative diagram:
\[\begin{tikzcd}[column sep=large]
	{\pi_0(R \cotimes E)} & {(\pi_0((R \cotimes E)^{hA})/(I_{tr}))^{\free}} & {\pi_0((R\cotimes E)^{\tau A})} \\
	& {(\pi_0((R \otimes E)^{hA})/(I_{tr}))^{\free}} & {\pi_0((R\otimes E)^{\tau A})^{\free}} & {\pi_0((R\otimes E)^{\tau A})} \\
	{\pi_0(R) \otimes \pi_0(E)} & {\pi_0(R) \otimes \left( \pi_0(E^{hA})/(I_{tr}) \right)^{\free}} & {\pi_0(R^{\tau A}) \otimes \pi_0(E^{\tau A})}
	\arrow["{{{{{\phi^m \cotimes t[A]}}}}}", from=1-1, to=1-2]
	\arrow["{{{{\overline{G_{R \cotimes E}}}}}}", from=1-2, to=1-3]
	\arrow[from=2-2, to=1-2]
	\arrow["{{{{\overline{G_{R \otimes E}}}}}}", from=2-2, to=2-3]
	\arrow[from=2-3, to=1-3]
	\arrow[from=2-4, to=1-3]
	\arrow[from=2-4, to=2-3]
	\arrow[from=3-1, to=1-1]
	\arrow["{{{{{\phi^m \otimes t[A]}}}}}"', from=3-1, to=3-2]
	\arrow[from=3-2, to=2-2]
	\arrow["{\can_R \otimes \overline{G_E}}"', from=3-2, to=3-3]
	\arrow[from=3-3, to=2-3]
	\arrow[from=3-3, to=2-4]
\end{tikzcd}\]
The left rectangle commutes because both vertical maps are completions with respect to the maximal ideal of $E_0$, by \cref{cor:pi-0-of-flat-tensor-E} and equation (\ref{eq:kunneth-plus-transfer}).

The middle upper square commutes as a consequence of the naturality of the map $\overline{G_R}$, as noted in \cref{def:the-map-F}. The middle lower square commutes by \cref{cor:lax-of-F}. 

Finally the triangles on the right all commute by functoriality of $(-)^{\free}$, in addition to the fact that $\pi_0(R^{\tau A})$, $\pi_0(E^{\tau A})$ and $\pi_0((R \cotimes E)^{\tau A})$ are all $p$-torsion free. 

Thus the lower path in diagram~(\ref{diag:after-pre-and-post}) is given by:

\begin{equation} \label{eq:lower-path}
\pi_0(R) \otimes \pi_0(E)  \xrightarrow[]{(\can_R \;\circ\; \phi^m) \,\otimes\, (\overline{G_E}\;\circ\; t[A])} \pi_0(R^{\tau A}) \otimes \pi_0(E^{\tau A}) \xrightarrow[]{ } \pi_0((R \cotimes E)^{\tau A})
\end{equation}
By \cref{cor:decomposition-of-frob},  $\overline{G_E} \,\circ\, t[A] = \phi_E$.
Moreover, by \cref{prop:composition-of-frobeni} to deduce that $\can_R \,\circ\, \phi^m = \phi_R$. 

Thus the morphism in (\ref{eq:lower-path}) identifies with the composite of the map $\pi_0(R^{\tau A}) \otimes \pi_0(E^{\tau A}) \to \pi_0((R \cotimes E)^{\tau A})$ with the morphism $\phi_R \otimes \phi_E$. This agrees with the composite along the upper path, completing the proof.

\end{proof}

\bibliographystyle{alpha}
\bibliography{./Bibl}

\end{document}